\documentclass[reqno]{asl}
\usepackage{amsfonts}
\usepackage{amssymb}

\usepackage{enumerate} 

\title[Central limit theorem on a lattice]{Application of semi-invariants to proof of the central limit theorem on a lattice}

\date{}

\keywords{Ising model; renormalization group; Gibbs measure; thermodynamic limit; weak dependence; semi-invariant;  cumulant; multivariate normal distribution; central limit theorem}

\subjclass{82B28 Renormalization group methods; 82B44 Disordered systems (random Ising models, random Schr\"odinger operators, etc.)}

\author{Farida Kachapova}
\address{School of Computer and Mathematical Sciences\\
Auckland University of Technology\\
Auckland, New Zealand}

\email{farida.kachapova@aut.ac.nz}

\author{Ilias Kachapov}
\address{Examination Academic Services\\ University of Auckland\\ Auckland, New Zealand}

\email{bill.kachapov@auckland.ac.nz}

\thanks{ }

\newpage

\newtheorem{theorem}{Theorem}[section]

\newtheorem{lemma}{Lemma}[section]

\newtheorem{example}{Example}[section]

\newtheorem{corollary}{Corollary}[section]

\theoremstyle{definition}
\newtheorem{definition}{Definition}[section]

\begin{document}

\begin{abstract}
Statistical mechanics describes interaction between particles of a physical system. Particle properties of the system can be modelled with a random field on a lattice and studied at different distance scales using renormalization group transformation. Here we consider a thermodynamic limit of Ising model with weak interaction and we use semi-invariants to prove that a random field transformed by renormalization group converges in distribution to an independent field with Gaussian distribution as the distance scale infinitely increases; it is a generalization of the central limit theorem to the Ising model.
\end{abstract}



\maketitle

\section{Introduction}
 
The central limit theorem plays an important part in probability theory and has applications in various fields. The classical form of the theorem considers a sequence of independent random variables and their normalized sums. Here we consider a sequence of random fields of weakly dependent random variables on a multi-dimensional integer lattice. We are interested in the limiting distribution of normalized sums of these variables, similar to the sums in the classical central limit theorem. Such problems about weakly dependent variables arise in the research of renormalization group in statistical mechanics.

The concept of renormalization group as a scale transformation was introduced and studied in works of Kadanoff \cite{Kad66}, \cite{Kad11}, Wilson and Kogut 
\cite{WK}, Sinai \cite{Si76},  Yin \cite{Yin}, and others. 

Originally renormalization  group  was defined in terms of Hamiltonian. \\Kashapov \cite{K80} derived for a high-temperature region a rigorous formalization of renormalization group in terms of the Hamiltonian of a Gibbs field. Later research on renormalization group was based on limit theorems of probability theory (see \cite{D68}, \cite{Si76} and \cite{C}). Sinai \cite{Si76} studied auto-model distributions, which are the distributions invariant under the renormalization group transformations, and he showed that Gaussian distribution is one of them. Malyshev \cite{MM91}, \cite{Ma80} developed technique of asymptotic estimation of semi-invariants for cluster expansions; this technique can also be applied to study properties of renormalization-group transformations.

In this paper we study the limits of distributions under the renormalization group transformations in Ising model (which is a mathematical model of a physical system with many particles). We show that under some conditions the limiting distribution in a high-temperature region is an independent Gaussian distribution. We modify and apply the techniques of Malyshev \cite{MM91}, \cite{Ma80} to estimate semi-invariants of a random field and we use these estimations to prove a generalization of the central limit theorem to the Ising model.

In Section 2 we introduce some concepts from probability theory and statistical physics and briefly prove some relevant lemmas. In Section 3 we state the main result of this paper: the central limit theorem for Ising model, with a brief discussion of its meaning. 

The rest of the paper develops techniques for proving the main theorem. In particular, in Section 4 we prove an inequality about the number of links in a set with a symmetric binary relation and apply it to estimate semi-invariants of a random field in Ising model with Gibbs measure. In subsection 5.1 we prove a series of lemmas, which lead to the direct proof of the main theorem in subsections 5.2 and 5.3. In particular, we find an expression for the limiting variances in Theorem \ref{theorem:variance} and show equality to 0 of all other limiting semi-invariants of the random field transformed by renormalization group. We complete the proof of the main theorem by applying Carleman's theorem to the limiting distribution.

\section{Main concepts}
\subsection{Semi-invariants}
Denote $E(X)$ the expectation of a random variable $X$. 
Semi-invariant is a generalization of the concepts of expectation and covariance. The following is a slight modification of the definition in \cite{MM91}, pg. 27-33.

\begin{definition}
Suppose $X_1,X_2,\ldots,X_m$
are random variables on the same probability space and $M=\lbrace 1,2,\ldots,m\rbrace $ is the set of their  indices.    
For any $S\subseteq M$, we denote $X_S=\prod_{i\in S} X_i $. We assume that the expectation of every such product is finite.

A \textit{semi-invariant} (or \textit{cumulant}) of random variables $X_1,X_2,\ldots,X_m $ is 
\begin{equation*} 
\langle X_1, X_2,\ldots,X_m \rangle=\sum_\alpha (-1)^{k-1}
 (k-1)!E(X_{S_1})\ldots E(X_{S_k}),
\label{Def} 
\end{equation*} 
where the sum is taken over all partitions  
$\alpha=\lbrace S_1,\ldots,S_k\rbrace $ 
of the set $M$. By a partition we mean a set of disjoint, non-empty subsets of $M$ such that their union equals $M$.

\textit{Notation}. If $I=(i_1,\ldots,i_m)$ is a sequence or a set of indices, we denote $\langle X^{,}_I\rangle=\langle X_{i_1},\ldots,X_{i_m}\rangle$.
\end{definition}

Semi-invariants characterize  the distribution and dependence of random variables.

\begin{example}
Suppose $X, X_1, X_2$ and $X_3$ are random variables. Denote $\mu$ the expectation of $X$ and $\sigma$ the standard deviation  of $X$. Then the following hold.

1) $\langle X\rangle=\mu$.
\medskip

2) $ \langle X_1,X_2\rangle =\langle X_1 X_2\rangle - \langle X_1\rangle \langle X_2\rangle=cov(X_1,X_2)$, the covariance of $X_1$ and $X_2$.
  
3) $  \langle X,X\rangle = \sigma^2 $, the variance of $X$. 
\begin{multline*}
4)\; \langle X_1,X_2,X_3 \rangle =\langle X_1 X_2 X_3 \rangle - \langle X_1\rangle   \langle X_2 X_3 \rangle - \langle X_2 \rangle  \langle X_1 X_3 \rangle - \langle X_3 \rangle  \langle X_1 X_2 \rangle 
\\
+2\langle X_1 \rangle \langle X_2 \rangle \langle X_3 \rangle.
\end{multline*}

5) $\langle X,X,X\rangle/\sigma^3$ equals the skewness of $X$.
\medskip

6) $\langle X,X,X,X\rangle /\sigma^4 $ equals the kurtosis of $X$.  
\end{example}

\begin{lemma}
1. A semi-invariant is a symmetrical and multi-linear functional on random variables.

2. If $ 0<n<m $ and two random vectors $(X_1,\ldots,X_n )$ and 
$( X_{n+1},\ldots,X_m )$ are independent of each other, then $\langle X_1,\ldots,X_n,X_{n+1},\ldots,X_m \rangle =0$. 

3. For set $M=\lbrace 1,2,\ldots,m\rbrace $:
\[E(X_M)=\langle X_M\rangle=\sum_\alpha \langle X^{,}_{S_1} \rangle\ldots\langle X^{,}_{S_k} \rangle,\]
where the sum is taken over all partitions  
$\alpha=\lbrace S_1,\ldots,S_k\rbrace $ 
of the set $M$. 
\label{lemma:sem}
\end{lemma} 
\begin{proof}
1. Follows from the definition of semi-invariants.

Parts 2 and 3 are proven in \cite{MM91}.
\end{proof}

The following is a well-known lemma about semi-invariants of normal distribution.
\begin{lemma} 
Suppose random variables $Y_1, Y_2,\ldots,Y_m$  have an independent multivariate normal distribution and $M=\lbrace 1,2,\ldots,m\rbrace$ is the set of their indices. 

1. If $k\geqslant 3$ and $i_1,\ldots,i_k\in M$, then $\langle Y_{i_1}, Y_{i_2},\ldots,Y_{i_k}\rangle=0$.
\medskip

2. If $i,j\in M$ and $i\neq j$, then $\langle Y_i, Y_j \rangle=0$. 
\label{lemma:normal1}
\end{lemma} 

\begin{lemma}
Suppose $Z_1, Z_2,\ldots,Z_m$  are independent random variables and each of them has the standard normal distribution. 

Suppose $\sigma_1>0,\sigma_2>0,\ldots,\sigma_m>0$  and $Y_i=\sigma_iZ_i$ $(i=1,2,\ldots,m).$ 
\smallskip

Then the random variables $Y_1, Y_2,\ldots,Y_m$ satisfy the Carleman's condition:
\begin{equation}
\sum_{n=1}^\infty \left(A_{2n}\right)^{-\frac{1}{2n}} =\infty,
\text{ where }A_k=\sum_{i=1}^m\langle Y_i^k \rangle.
\label{eq:Carleman}
\end{equation}

\label{lemma:normal2} 
\end{lemma}  
\begin{proof}
Clearly, $\langle Z_i^{2n}\rangle=(2n-1)!!$ for $n=1,2,\ldots,m$. 
\medskip

Denote $\sigma=\max \{\sigma_1,\ldots\sigma_m\}$. For $i=1,\ldots,m$:
\[\langle Y_i^{2n}\rangle=\langle \sigma_i^{2n} Z_i^{2n}\rangle=\sigma_i^{2n}\langle  Z_i^{2n}\rangle\leqslant\sigma^{2n}(2n-1)!!=\sigma^{2n}\dfrac{(2n)!}{2^n n!}.\]

By Stirling formula,
\[n!=\sqrt{2\pi n}\left(\dfrac{n}{e} \right)^n\theta_n,\text{ where }1<\theta_n<e.\]
\begin{multline*}
\text{So }A_{2n}=\sum_{i=1}^m\langle Y_i^{2n}\rangle
\leqslant m\sigma^{2n}\dfrac{(2n)!}{2^n n!}
\\
= m\sigma^{2n}\dfrac{\sqrt{2\pi 2n}\left( \dfrac{2n}{e}\right)^{2n}\theta_{2n}}{2^n\sqrt{2\pi n}\left( \dfrac{n}{e}\right)^{n}\theta_{n}}=\sqrt{2}m \cdot\dfrac{\theta_{2n}}{\theta_{n}}\left(\dfrac{2\sigma^2}{e} \right)^n n^n\leqslant c_1\left(\dfrac{2\sigma^2}{e} \right)^n n^n
\end{multline*}
for some positive constant $c_1$ (depending only on $m$), since $\dfrac{\theta_{2n}}{\theta_n}<e$. Next,
\[\left(A_{2n}\right)^{-\frac{1}{2n}}\geqslant \left(c_{1}\right)^{-\frac{1}{2n}}\left(\dfrac{2\sigma^2}{e} \right)^{-\frac{1}{2}} n^{-\frac{1}{2}}\geqslant \dfrac{c_2}{\sqrt{n}}\]
for some constant $c_2>0$.
Therefore 
\[\sum_{n=1}^\infty \left(A_{2n}\right)^{-\frac{1}{2n}}\geqslant \sum_{n=1}^\infty \dfrac{c_2}{\sqrt{n}}= 
\infty.\]
\end{proof}

\begin{lemma}
Suppose $M=\lbrace 1,2,\ldots,m\rbrace$ and random variables $X_1,X_2,\ldots,\\X_m$  satisfy the following conditions:
\begin{list}{•}{•}
\item
for $k\geqslant 3$ and $i_1,\ldots,i_k\in M$, $\langle X_{i_1}, X_{i_2},\ldots,X_{i_k}\rangle=0$;
\medskip
\item
for $i,j\in M$, $i\neq j$,  $\langle X_i, X_j \rangle=0$.
\end{list}
Then $X_1,X_2,\ldots,X_m$ have an independent multivariate normal distribution.  
\label{lemma:Carleman} 
\end{lemma} 
\begin{proof}
Denote $\mu_i=\langle X_i\rangle$, $\sigma_i^2=\langle X_i,X_i\rangle$ and $V_i=X_i-\mu_i$ $(i=1,2,\ldots,m)$. Then $\langle V_i\rangle=0$ for each $i=1,2,\ldots,m$. Other corresponding semi-invariants are the same for the random vectors $(X_1,X_2,\ldots,X_m)$ and $(V_1,V_2,\ldots,V_m)$.

Consider independent random variables $Z_1, Z_2,\ldots,Z_m$, where each $Z_i$ has the standard normal distribution, and denote $Y_i=\sigma_iZ_i$. By Lemma \ref{lemma:normal1}, corresponding semi-invariants are the same for the random vectors $(V_1,V_2,\ldots,V_m)$ and $(Y_1,Y_2,\ldots,Y_m)$. Semi-invariants uniquely determine moments. So  corresponding moments are also the same for the random vectors $(V_1,V_2,\ldots,V_m)$ and $(Y_1,Y_2,\ldots,Y_m)$. By Lemma \ref{lemma:normal2}, these moments satisfy the Carleman's condition and by Carleman's theorem, $(V_1,V_2,\ldots,V_m)$ and $(Y_1,Y_2,\ldots,Y_m)$ have the same probability distribution. 

Therefore, the random variables $V_1,V_2,\ldots,V_m$  have an independent multivariate normal distribution. Since each $X_i=V_i+\mu_i$, the random variables $X_1,X_2,\ldots,X_m$ also have an independent multivariate normal distribution.
\end{proof}

\subsection{Ising model}
\label{subsection:Ising model}

For the rest of the paper we fix a natural number $\nu\geqslant 1$ and consider a $\nu$-dimensional integer lattice:
\[\mathbb{Z}^\nu= \{(t_1,\ldots,t_\nu)\mid t_i\in \mathbb{Z}, i=1,\ldots,\nu\}\]
with the distance between any two points given by:
\[\rho(s,t)=\sum_{i=1}^{\nu} |s_i-t_i|.\]

$R=\left\lbrace \{s,t\}\mid s,t\in \mathbb{Z}^\nu \;\&\; \rho(s,t)=1\right\rbrace$. $R$ is the set of all pairs of neighbouring nodes in the lattice $\mathbb{Z}^\nu$.

$\Omega=\left\lbrace 
\omega\mid\omega:\mathbb{Z}^\nu\rightarrow \{-1,1\}\right\rbrace.$

We associate with each $t\in\mathbb{Z}^\nu$ a function $Q_t:\Omega\rightarrow \{-1,1\}$ such that for each $\omega\in\Omega$:
\[Q_t(\omega)=\omega(t).\]

The following is the definition of an Ising model with no interaction with external fields and with a constant strength of interaction along the lattice; it is a particular case of the definition in \cite{MM91}.   

\begin{definition}
We fix a real number $\lambda$ and a natural number $N\geqslant 1$.

An \textit{Ising model} with parameters $\lambda$ and $N$ is a triple of objects $(\Lambda_N,\Omega_N,U_N)$, which are defined as follows.

1. $\Lambda_N=\left\lbrace t\in \mathbb{Z}^\nu\mid \text{ for each }i=1,2,\ldots,\nu,|t_i|\leqslant N
\right\rbrace.$
Thus, $\Lambda_N$ is a cube in $\mathbb{Z}^\nu$.
\smallskip

2. $\Omega_N=\left\lbrace 
\delta\mid\delta:\Lambda_N\rightarrow \{-1,1\}\right\rbrace.$ Elements of $\Omega_N$ are called \textit{configurations} or \textit{states}.

3. Denote $R_N=\left\lbrace \{r,s\}\in R\mid r,s\in \Lambda_N \right\rbrace$. $R_N$ is the set of all pairs of neighbouring nodes in the cube $\Lambda_N$.

Function $U_N :\Omega\rightarrow \mathbb{R}$ is defined by the following:
\[ U_N (\omega)=-\lambda\sum_{\{r,s\}\in R_N}\omega(r)\omega(s).\]

This completes the definition of Ising model.
\label{def:Ising model}
\end{definition}

The Ising model describes a physical system with many particles represented by nodes of the cube $\Lambda_N$ in the integer lattice; $\Omega_N$ is the set of all states of the system, function $U_N$ characterizes the interaction energy of the system and $|\lambda|$ is proportional to the inverse temperature of the system. The parameter $\lambda$ also characterizes the strength of interaction between particles, and we assume that only neighbouring particles interact. 

The Ising model with $\lambda =0$ describes a physical system with no interaction between its elements, e.g. ideal gas. The Ising model with $\lambda>0$ describes the ferromagnetic system and the Ising model with $\lambda<0$ describes the anti-ferromagnetic system.

\subsection{Gibbs Measure}

The definition of Gibbs measure can be found in \cite{D68}. Since it is important for our paper, we also provide the definition.
\begin{definition}
For the Ising model we define the \textit{associated probability space}\\ $(\Omega, \Sigma_N, P_{\lambda,N})$ as follows.

1. The sample space is the set $\Omega$ as defined before.

2. The sigma-algebra $\Sigma_N$ of events consists of all finite unions of the sets:
\[A_\delta=\{\omega\in\Omega\mid(\forall t\in\Lambda_N)(\omega(t)=\delta(t))\},\delta\in\Omega_N.\]

3. $\overline{U}_N :\Omega_N\rightarrow \mathbb{R}$ is defined by the following:
\[\overline{U}_N (\delta)=U_N (\omega)\text{ for } \omega\in A_\delta.\]
The definition is valid because $U_N (\omega_1)=U_N (\omega_2)$ for any $\omega_1,\omega_2\in A_\delta$.

4. The probability of event $A_\delta$  is defined by:
\begin{equation}
P_{\lambda,N}(A_\delta) =\dfrac{1}{\Xi} e^{-\overline{U}_N (\delta)}, \text{ where }\Xi=\sum_{\delta'\in\Omega_N} e^{-\overline{U}_N(\delta')}. 
\label{prob1}
\end{equation}

This generates the probability measure $P_{\lambda,N}$ on all events in $\Sigma_N$, which is called 
\textit{Gibbs measure} on the cube $\Lambda_N$. 
The formula \eqref{prob1} ensures that $P_{\lambda,N}(\Omega)=1$, since $\Omega=\cup_{\delta\in\Omega_N}A_\delta$. This completes the definition of the associated probability space.
\end{definition}

$\overline{U}_N(\delta)$ characterizes the energy of configuration $\delta$ of the cube $\Lambda_N$. 
We denote $\langle \cdot,\ldots,\cdot\rangle_{\lambda,N}$ the semi-invariants with respect to the Gibbs measure $P_{\lambda,N}$.
\medskip

Clearly, $\{Q_t\mid t \in \Lambda_N\}$ is a random field on the associated probability space.

\begin{lemma} Suppose $ \lambda=0 $. Then the following hold.
 
1. $P_{0,N}\left(Q_{t_1}=a_1,\ldots,Q_{t_m}=a_m\right)=
2^{-m}$
for any $a_1,\ldots,a_m\in\{1,-1\} $ and distinct points 
$t_1,\ldots, t_m\in\Lambda_N $.
\medskip

2. $\{Q_t\mid t \in \Lambda_N\}$ is an independent random field with respect to $P_{0,N}$. 
\medskip

3. The distribution of the random field $\{Q_t\mid t \in \Lambda_N\}$  with respect to the Gibbs measure $P_{0,N}$ does not depend on $N$.
\medskip

4. For any $t\in \mathbb{Z}^\nu$, $\langle Q_t\rangle_{0}=0$.
\medskip

5. For any $m\geqslant 1$ and distinct $  t_1,\ldots, t_m \in \mathbb{Z}^\nu$, $\langle Q_{t_1}\cdot\ldots\cdot Q_{t_m}\rangle_{0}=0$.
\medskip

6. If $m$ is odd and $  t_1,\ldots, t_m \in \mathbb{Z}^\nu$, then $\langle Q_{t_1}\cdot\ldots\cdot Q_{t_m}\rangle_{0}=0$.
\medskip

Due to part 3, semi-invariants of variables $Q_t (t\in\Lambda_N)$ with respect to measure $P_{0,N}$ do not depend on $N$ and we denote them with $\langle \cdot,\ldots,\cdot\rangle_{0}$.
\label{lemma:zero_lambda}
\end{lemma}

\begin{proof}
Suppose $\lambda=0 $. Then for any $\delta\in \Omega_N$, $\overline{U}_N(\delta)=0$. There are $|\Omega_N|=2^{|\Lambda_N|}$ configurations in $\Omega_N$, where $|\Lambda_N|=(2N+1)^\nu$ is the number of points in the cube $\Lambda_N$.

So for any $\delta\in \Omega_N$: 
\[P_{0,N}(A_\delta)=\dfrac{1}{|\Omega_N|}=\dfrac{1}{2^{|\Lambda_N|}}.\]

For any $t\in\Lambda_N$:
\[P_{0,N}(Q_t=1)=
P_{0,N}\left(\bigcup_{\{\delta:\delta(t)=1\}}A_\delta\right)
=2^{|\Lambda_N|-1}\cdot
\dfrac{1}{2^{|\Lambda_N|}}=\dfrac{1}{2}.\]

Thus, for any $t\in\Lambda_N$:
\begin{equation} 
 P_{0,N}(Q_t=1)
=P_{0,N}(Q_t=-1)=\dfrac{1}{2}.
\label{Eq3}
\end{equation} 

Consider $a_1,\ldots,a_m\in\{1,-1\} $ and distinct points 
$t_1,\ldots, t_m\in\Lambda_N $.
\begin{multline*}
P_{0,N}\left(Q_{t_1}=a_1,\ldots,Q_{t_m}=a_m\right)=
\\
P_{0,N}\left(\bigcup\left\lbrace A_\delta\mid\delta(t_1)=a_1,\ldots,
\delta(t_m)=a_m\right\rbrace \right)
\\
=2^{|\Lambda_N|-m}\cdot
\dfrac{1}{2^{|\Lambda_N|}}=2^{-m}=\prod_{i=1}^m P_{0,N}(Q_{t_i}=a_i)
\end{multline*} 
by \eqref{Eq3}. This proves parts 1 and 2 of the lemma, and part 3 follows from part 1.

4. It follows from \eqref{Eq3}.

5. By part 2, $Q_{t_1},\ldots, Q_{t_m}$ are independent, so by part 4,
\[\langle Q_{t_1}\cdot\ldots\cdot Q_{t_m}\rangle_0= \langle Q_{t_1}\rangle_0\cdot\ldots\cdot \langle Q_{t_m}\rangle_0=0.\]

6. If $t_1=t_2$, then $Q_{t_1}\cdot Q_{t_2}=1$  and $Q_{t_1}\cdot Q_{t_2}\cdot Q_{t_3}\cdot\ldots\cdot Q_{t_m}=Q_{t_3}\cdot\ldots\cdot Q_{t_m}$. 

Therefore we can remove pairs of variables with equal indices from the product:
\[Q_{t_1}\cdot Q_{t_2}\cdot\ldots\cdot Q_{t_m}=Q_{t_{i_1}}\cdot  \ldots\cdot Q_{t_{i_k}},\] 
where the remaining $t_{i_1},\ldots,t_{i_k}$ are distinct; the resulting product contains at least one multipler, since $m$ is odd. So $\langle Q_{t_1}\cdot Q_{t_2}\cdot\ldots\cdot Q_{t_m}\rangle_0=\langle Q_{t_{i_1}}\cdot  \ldots\cdot Q_{t_{i_k}}\rangle_0=0$ by part 5.
\end{proof}  

\begin{lemma} 
Suppose $f:\mathbb{Z}^\nu\rightarrow\mathbb{Z}^\nu$ is a bijection; $t_1,t_2,\ldots,t_m\in\mathbb{Z}^\nu$ and $s_i=f(t_i)$ for $i=1,2,\ldots,m$. Then the random vectors $(Q_{t_1},Q_{t_2},\ldots,Q_{t_m})$ and $(Q_{s_1},Q_{s_2},\ldots,Q_{s_m})$ have the same distribution with respect to measure $P_{0,N}$.
\label{lemma:translation_1}
\end{lemma}

\begin{proof}
Consider $a_1,a_2,\ldots,a_m\in \{-1,1\}$. If there is a pair $t_i=t_j$ with $a_i\neq a_j$, then $s_i=s_j$ and 
\[P_{0,N}\left(Q_{t_1}=a_1,\ldots,Q_{t_m}=a_m\right)
=0=P_{0,N}\left(Q_{s_1}=a_1,\ldots,Q_{s_m}=a_m\right).\]

Otherwise let us write the points $t_1,\ldots,t_m$ without repetitions: $t_{i_1},\ldots,t_{i_k}$. Then the points $s_{i_1},\ldots,s_{i_k}$ are also distinct. So by Lemma \ref{lemma:zero_lambda}.1) we have: 
\begin{multline*}
P_{0,N}\left(Q_{t_1}=a_1,\ldots,Q_{t_m}=a_m\right)
=P_{0,N}\left(Q_{t_{i_1}}=a_{i_1},\ldots,Q_{t_{i_k}}=a_{i_k}\right)=2^{-k}
\\
=P_{0,N}\left(Q_{s_{i_1}}=a_{i_1},\ldots,Q_{s_{i_k}}=a_{i_k}\right)
=P_{0,N}\left(Q_{s_1}=a_1,\ldots,Q_{s_m}=a_m\right).
\end{multline*} 
\end{proof}
 
\subsection{Thermodynamic limit}

An Ising model $(\Lambda_N,\Omega_N,U_N)$ has two parameters $\lambda$ and $N$ and this model generates the associated probability space $(\Omega,\Sigma_N,P_{\lambda,N})$. Let us see what happens when $N\rightarrow\infty$. Clearly, the finite cube $\Lambda_N$ transforms into the lattice $\mathbb{Z}^\nu$ and $\Omega_N$ transforms into $\Omega$. 

For any finite subset $T$ of $\mathbb{Z}^\nu$ we denote \[Q_T=\prod_{t\in T} Q_t.\]

\begin{theorem}
There exists a positive constant $C_\nu$ (depending only on $\nu$) such that if $|\lambda|<C_\nu$, then the following hold.
\[1. \text{ For any }t_1,t_2,\ldots,t_m\in\mathbb{Z}^\nu\text{ the limit }\lim_{N\rightarrow\infty}\langle Q_{t_1},Q_{t_2},\ldots,Q_{t_m}
\rangle_{\lambda,N} \text{ exists}.\]
\[2.\text{ For any finite subset }T\text{ of }\mathbb{Z}^\nu\text{ the limit }f(Q_T)=\lim_{N\rightarrow\infty}\langle Q_T\rangle_{\lambda,N}\text{ exists.}\]
\label{theorem:therm_limit}
\end{theorem}

Proof will be given in subsection \ref{subsection_proof1}.

\begin{definition}
Assume $|\lambda|<C_\nu$. The \textit{limiting probability space} $(\Omega,\Sigma,P_\lambda)$ is defined as follows.

1. As before, $\Omega=\left\lbrace 
\omega\mid\omega:\mathbb{Z}^\nu\rightarrow \{-1,1\}\right\rbrace.$ 

2. The sigma-algebra $\Sigma$ of events consists of all countable unions of the sets: 
\[M_{T,A}=\{
\omega\in\Omega\mid\omega(t_1)=a_1,\ldots,\omega(t_m)=a_m\}\]
for all pairs of sets $T=\{t_1,t_2,\ldots,t_m\}\subset\mathbb{Z}^\nu$ and $A=\{a_1,a_2,\ldots,a_m\}\subseteq \{-1,1\}$, $m\geqslant 0$. It is sufficient to take only $A\subseteq \{-1,1\}$ because for other $A\subset\mathbb{R}$, $M_{T,A}=\varnothing$.

3. For any pair of $T=\{t_1,t_2,\ldots,t_m\}\subset\mathbb{Z}^\nu$ and $A=\{a_1,a_2,\ldots,a_m\}\subseteq \{-1,1\}$ we define the probability of $M_{T,A}$ by the following:
\begin{equation}
P_\lambda(M_{T,A})=\frac{(-1)^k}{2^m}\sum_{T'\subseteq T}f(Q_{T'})\prod_{\{i:t_i\in T\setminus T'\}}a_i,
\label{eq:limit_prob}
\end{equation}
where $f$ is defined in Theorem \ref{theorem:therm_limit}.2) and $k$ is the number of $a_i\in A$ that equal -1.

The formula \eqref{eq:limit_prob} is given in \cite{MM91}; it generates the probability measure  $P_\lambda$ on all events in $\Sigma$, which is called the \textit{limiting Gibbs measure}.

This completes the definition of the limiting probability space.
\label{def:probability_space}
\end{definition}

\begin{definition}
The \textit{thermodynamic} or \textit{macroscopic limit} of Ising model with parameter $\lambda$ is the lattice $\mathbb{Z}^\nu$ together with the limiting probability space as defined in Definition \ref{def:probability_space}.
\end{definition}

Clearly, $\{Q_t\mid t\in\mathbb{Z}^\nu\}$ is a random field on the limiting probability space. We denote $\langle \cdot,\ldots,\cdot\rangle_{\lambda}$ the semi-invariants with respect to the limiting Gibbs measure $P_\lambda$.

\begin{theorem}
Suppose $|\lambda|<C_\nu$, where  $C_\nu$ is the constant from Theorem \ref{theorem:therm_limit}. Then the following hold.

1. For any finite subset $T$ of $\mathbb{Z}^\nu$:
\[\lim_{N\rightarrow\infty}\langle Q_T\rangle_{\lambda,N}=\langle Q_T\rangle_{\lambda}.\]

2. For any $t_1,t_2,\ldots,t_m\in\mathbb{Z}^\nu$:
\[\lim_{N\rightarrow\infty}\langle Q_{t_1},Q_{t_2},\ldots,Q_{t_m}
\rangle_{\lambda,N}=\langle Q_{t_1},Q_{t_2},\ldots,Q_{t_m}
\rangle_{\lambda}.\]
\label{theorem:limits}
\end{theorem}
\begin{proof}

1. The proof uses \eqref{eq:limit_prob} and is similar to the corresponding proof in \cite{MM91}, pg. 2 (a probability measure and its semi-invariants can be defined through each other). 

2. This follows from part 1 by the definition of semi-invariant.
\end{proof}

\subsection{Renormalization group}
  The following concept was introduced by Kadanoff \cite{Kad66}.
\begin{definition}
Fix a natural number $k > 1$ and a real number $ \alpha\geqslant\nu $.

1. Define a mapping $G_k: \mathbb{Z}^\nu\to \mathbb{Z}^\nu$ 
  as follows:
\begin{equation} \nonumber
 G_k(t_1,t_2,\ldots,t_\nu)=\left( \left[ \dfrac{t_1}{k}\right],
\left[\dfrac{t_2}{k}\right],\ldots,\left[\dfrac{t_\nu}{k}\right]\right),   
\end{equation}
where $[x]$ denotes the integer part of a real number $x$.
  For any point $ \tau \in \mathbb{Z}^\nu $ there are $ k^\nu $
  points that are mapped into $ \tau $   by $ G_k $.

2. A \emph{renormalization group} (with parameters $ k $ 
 and $\alpha$) is a transformation that assigns to each random field  $\{X_t\mid t \in \mathbb{Z}^\nu\}$ another random field  $\{Y^{(k)}_\tau\mid \tau\in \mathbb{Z}^\nu\}$ given by: 
\begin{equation}	
\label{RG}
Y^{(k)}_\tau=k^{-\frac{\alpha}{2}} \sum_{t \in G_k^{-1}(\tau)} X_t.
\end{equation}  
\end{definition}

The renormalization group is a scaling transformation. It allows to study the physical system at different distance scales, such as atomic and molecular levels. Details of its physical interpretation can be found in \cite{Kad11}.

We are interested in the distribution of the result $ Y^{(k)}_\tau$ of the renormalization group transformation of the field $\{Q_t\mid t\in\mathbb{Z}^\nu\}$.

\section{The central limit theorem for Ising model}

First we introduce some notations. We use letters $b,c,d,\ldots$ for finite ordered sequences (or in short, sequences). For a sequence $b=(T_1,\ldots,T_n)$ we denote $|b|=n$.

\begin{definition}
1. A \textit{family} (of elements of a set $ \mathfrak{A}$) is a set of pairs \\$\alpha=\{(T_1,n_1), \ldots,(T_m,n_m)\}$, where $T_1,\ldots,T_m$ are distinct elements of $ \mathfrak{A}$ and $n_i\geqslant 1$ for each $i=1,\ldots,m$. 

2. The number $n_i$ is called the \textit{multiplicity} of element $T_i$ in the family $\alpha$.

3. We denote the length of the family $\alpha$ as 
$|\alpha|=n_1+n_2+\ldots+n_m$ and \[\alpha!=n_1!\cdot n_2!\cdot\ldots\cdot n_m! \]

4. When $T_1, \ldots,T_m$ are sets we denote:
\[\tilde{\alpha}=\bigcup_{i=1}^mT_i.\]
\end{definition} 

We use letters $\alpha,\beta,\gamma,\ldots$ for families.
The same elements $T_1,\ldots,T_n\in\mathfrak{A}$ can be represented as a sequence or a family.

\begin{definition}
1. Any sequence $b=(T_1,\ldots,T_n)$ \textit{reduces} to a family:
\[\{(\overline{T}_1,n_1), \ldots,(\overline{T}_q,n_q)\},\]
where $\overline{T}_1,\ldots,\overline{T}_q$ are the elements 
$T_1,\ldots,T_n$ written without repetitions, and each $n_i$ is the number of times that $\overline{T}_i$ is repeated in $\alpha$; $n_1+\ldots+n_q=n$.

For each family $\alpha$ of length $n$ there are $\dfrac{n!}{\alpha !}$ sequences that reduce to $\alpha$.
\medskip

2. For any family $\alpha=\{(B_1,n_1), \ldots,(B_q,n_q)\}$ we can define an \textit{associated sequence}:
\[\widehat{\alpha}=\left( \underbrace{B_1,\ldots,B_1}_{n_1\texttt{ times}} ,\ldots, \underbrace{B_q,\ldots,B_q}_{n_q\texttt{ times}}\right).\]
Then $\widehat{\alpha}$ reduces to $\alpha$.
\end{definition}

The set $R$ of all pairs of neighbouring nodes in $\mathbb{Z}^\nu$ was defined in subsection 2.2. Denote $R^*=\{\gamma\mid\gamma$ is a family of elements of $R\}$.

\begin{definition}
Consider a sequence $b=(t_1,t_2\ldots,t_m)$ of points in $\mathbb{Z}^\nu$ and a
family $\gamma\in R^*$.

1. Consider also the associated sequence $\widehat{\gamma}=(A_1,A_2,\ldots,A_n)$.

The \textit{associated graph} of $b$ and $\gamma$ is defined as follows:

its edges are $A_1,A_2,\ldots,A_n$;

the set of its vertices is $\{t_1,t_2\ldots,t_m\}\cup\bigcup_{i=1}^n A_i$.

2. We say that the family $\gamma$ \textit{connects} the sequence $b$ if the associated graph of $b$ and $\gamma$ is connected.
\end{definition}

\begin{definition}
For any $A=\{r, s\}\in R$ we denote $\Phi_A= Q_r Q_s$; the function $\Phi_{A}$ represents the interaction between the neighbours $r$ and $s$. 

The set $\{\Phi_{A}\mid A\in R\}$ is called the \textit{potential}.  
\end{definition}

\textit{Notation}. Suppose $b=(t_1,t_2\ldots,t_m)$ is a sequence of points in $\mathbb{Z}^\nu$, $\gamma\in R^*$ and the associated sequence 
$\widehat{\gamma}=(A_1,A_2, \ldots,A_n).$ Denote
\[\langle Q^,_b,\Phi^\backprime_\gamma\rangle_0=
\langle Q_{t_1},Q_{t_2},\ldots,Q_{t_m},
\Phi_{A_1},\Phi_{A_2},\ldots, \Phi_{A_n}\rangle_0.\]

Denote $\bar{0}$ the origin in the lattice $\mathbb{Z}^\nu$: $\bar{0}=(0,0,\ldots,0)$.

\begin{theorem} (Main Theorem)
Consider the thermodynamic limit of the \\Ising model with parameter $\lambda$. 

Suppose a renormalization group with parameters $k$ and  $\alpha$  transforms the random field $\{Q_t\mid t\in\mathbb{Z}^\nu\}$ into a random field $\lbrace   Y^{(k)}_\tau \mid\tau\in \mathbb{Z}^\nu \rbrace$. There exists a positive constant $C$ such that for any $|\lambda|< C$ the following hold.
\medskip

1. Suppose $\alpha>\nu$. Then the field $Y^{(k)}_\tau \rightarrow 0$ in mean square as $ k\rightarrow \infty$.  
\medskip

2. Suppose $\alpha=\nu $. Then as $ k\rightarrow \infty$, the field $\lbrace Y^{(k)}_\tau \mid\tau\in \mathbb{Z}^\nu \rbrace$ converges in distribution to an independent field with Gaussian distribution (i.e. any finite subset of the field has a multivariate normal distribution). Each of the variables of the limiting field has 0 expectation and the positive variance given by:
\[V=1+\sum_{n=1}^{\infty}
\lambda^nV_n,\text{ where each }V_n=\sum_{\substack{\gamma\in R^*,
\\
|\gamma|=n,
\\
\gamma\text{ connects }(\bar{0})
}} 
\sum_{\substack{t\in \tilde{\gamma},
\\
\gamma\text{ connects }(\bar{0},t),
\\
t\neq \bar{0}
}}
\dfrac{1}{\gamma!}
\langle Q_{\bar{0}},Q_t,\Phi^\backprime_\gamma\rangle_0.\]
\label{theorem:main}
\end{theorem}

Proof will be given in Section 5.

The classical central limit theorem considers a sequence $X_1, X_2,X_3,\ldots$ of independent, identically distributed random variables with finite variances and states that as $n\rightarrow \infty$, their normalized sum $\dfrac{1}{\sqrt{n}}\sum_{i=1}^n(X_i-\mu)$ converges in distribution to a normal random variable. 

Theorem \ref{theorem:main} can be considered as a generalization of the classical central limit theorem, in some sense. Instead of a sequence of random variables we have a sequence of random fields $\lbrace Y^{(k)}_\tau \mid\tau\in \mathbb{Z}^\nu \rbrace $ on a multi-dimensional integer lattice. The variables $Q_t$ are identically distributed (each has Bernoulli distribution). They are weakly dependent because $|\lambda|<C$ and $\lambda$ characterizes the strength of the interaction. 

Like the classical central limit theorem, Theorem \ref{theorem:main}.2) also considers a normalized sum, that is the sum $\sum_{t \in G_k^{-1}(\tau)} Q_t$ divided by square root of the number $k^\nu$ of addends in the sum. Theorem \ref{theorem:main}.2) states convergence in distribution and that the limiting distribution is normal but in this case it is the distribution of an independent normal field. In other words, Theorem \ref{theorem:main}.2) states: in systems with weak interaction the distribution of the normalized sums over big regions is approximately independent and normal.

\section{Estimation of dependencies}
\label{section:graphs}
\subsection{Estimation Theorem}

The proof of the central limit theorem in Section \ref{section:proof} is based on estimations of semi-invariants. In this section we prove an inequality (Theorem \ref{theorem:estimation}), which will be applied to estimating semi-invariants. 

In this section we fix a set $\mathfrak{A}$ and a reflexive, symmetric binary relation on $\mathfrak{A}$. If $A, B\in \mathfrak{A}$ are in this relation, then we say that $A$ and $B$ are \textit{linked}. Thus, $A$ is always linked to $A$ (reflexivity). If $A$ is linked to $B$, then $B$ is linked to $A$ (symmetry).

We assume that there exists a constant $L$ such that each element $A\in \mathfrak{A}$ is linked to at most $L$ elements in $\mathfrak{A}$. We fix this constant $L$.

\begin{definition}
For a family $\alpha=\{(T_1,n_1), \ldots,(T_m,n_m)\}$ of elements of $\mathfrak{A}$ and any $i=1,2,\ldots,m$ we denote 
\[\upsilon(T_i)=\upsilon_{\alpha}(T_i)=\sum_{\{j\in M:T_j\text{ is linked to }T_i\}} n_j.\]
\end{definition}
This sum has at most $L$ addends. 

The following property of natural logarithm will be used in proofs later: 
\begin{equation}
\text{for any }x>0:\;\ln(1+x)< x.
\label{eq:logarithm}
\end{equation}

In probability terms, the following theorem estimates the number of pairwise dependencies in a random field. This theorem is similar to the theorem in \cite{MM91}, pg. 59, estimating the number of intersections. The theorem in \cite{MM91} has a stronger conclusion. Our theorem has a simpler proof by induction and more general conditions: it is stated for an abstract symmetric binary relation while the theorem in \cite{MM91} is stated for a particular binary relation (when two sets have a non-empty intersection in a countable metric space). 

\begin{theorem} (Estimation Theorem)
Denote $C_L = \ln L +L^2$. For any family $\alpha=\{(T_1,n_1), \ldots,(T_m,n_m)\}$, $m\geqslant 1$, the following inequality holds:
\[C_L\sum_{i=1}^m n_i +\sum_{i=1}^m n_i \ln n_i>\sum_{i=1}^m n_i\ln \upsilon_\alpha(T_i).\]
\label{theorem:estimation}
\end{theorem}
\begin{proof} Denote $M=\{1, 2,\ldots,m\}$. In the proof for brevity we will write $C$ for $C_L$ and $\upsilon_i$ for $\upsilon_\alpha(T_i)$. Denote
\[f(\alpha)=C\sum_{i=1}^m n_i +\sum_{i=1}^m n_i \ln n_i \;\text{ and }\;g(\alpha)=\sum_{i=1}^m n_i\ln \upsilon_i.\] 

Thus, we need to prove:
\begin{equation}
f(\alpha)>g(\alpha).
\label{T2}
\end{equation}

The proof is by induction on $k=\max\{n_1,\ldots,n_m\}$.  Clearly, for any $i\in M:$
\begin{equation}
\upsilon_i\leqslant kL.
\label{eq:est_epsilon}
\end{equation}

\textit{Basis of induction:} $k=1$.
Then each $n_i=1$ and $\upsilon_i\leqslant L$. So 
\begin{multline*}
f(\alpha)=C\sum_{i=1}^m 1 +\sum_{i=1}^m 1 \ln 1=Cm=m(\ln L+L^2)>m\ln L=\sum_{i=1}^m\ln L
\\
\geqslant \sum_{i=1}^m\ln \upsilon_i=g(\alpha).
\end{multline*}

\textit{Inductive step.} 
Assume that \eqref{T2} holds for $k$ $(k\geqslant 1).$

Consider a family $\alpha=\{(T_1,n_1), \ldots,(T_m,n_m)\}$ with $\max\{n_1,\ldots,n_m\}=k+1.$ For any $i\in M$ denote $n'_i=\min\{n_i,k\}.$ Define a new family $\alpha'=\{(T_1,n'_1), \ldots,(T_m,n'_m)\}$ and denote $\upsilon'_i=\upsilon_{\alpha'}(T_i)$.
\medskip

For this family $\max\{n'_1,\ldots,n'_m\}=k$ and for any $i\in M,$ 
\[\upsilon'_i=\sum_{\{j\in M:T_j\text{ is linked to }T_i\}} n'_j.\]

So by the inductive assumption:
\begin{equation}
f(\alpha')>g(\alpha').
\label{eq:ind_assumption}
\end{equation}

It is sufficient to prove:
\begin{equation}
f(\alpha)-f(\alpha')>g(\alpha')-g(\alpha').
\label{eq:ind_step}
\end{equation}

Adding the inequalities \eqref{eq:ind_assumption} and \eqref{eq:ind_step} we get $f(\alpha)>g(\alpha).$
\begin{center}
Proof of \eqref{eq:ind_step}
\end{center}
\begin{multline*}
f(\alpha)-f(\alpha')=C\sum_{i=1}^m n_i +\sum_{i=1}^m n_i \ln n_i-\left(C\sum_{i=1}^m n'_i +\sum_{i=1}^m n'_i \ln n'_i\right)
\\
= C\sum_{i=1}^m(n_i-n'_i) +\sum_{i=1}^m (n_i \ln n_i-n'_i \ln n'_i).
\end{multline*}

Denote $I=\{i\in M\mid n_i=k+1\}.$

For $i\in I:$ $n_i=k+1,$ $n'_i=k$, $n_i-n'_i=1$ and 
\begin{multline*}
n_i\ln n_i-n'_i\ln n'_i=
(k+1)\ln(k+1)-k\ln k=\ln(k+1)+k[\ln(k+1)-\ln k]
\\
=\ln(k+1)+k\ln\left(1+\frac{1}{k}\right).
\end{multline*}

For $i\in M\setminus I:$ $n'_i=n_i,$ $n_i-n'_i=0$ and $n_i\ln n_i-n'_i\ln n'_i=0$. 

Denote $|I|$ the number of elements in the set $I$. So 
\begin{multline*}
f(\alpha)-f(\alpha')=C\sum_{i\in I}1 +\sum_{i\in I}\left[\ln(k+1)+k\ln\left(1+\frac{1}{k} \right)\right]
\\
>C\sum_{i\in I}1 +\sum_{i\in I} \ln(k+1)
=|I|C +|I|\ln(k+1)
=|I| \left[C+\ln(k+1)\right]
\\
=|I| \left[\ln L+L^2+\ln(k+1)\right]
=|I| \left[\ln((k+1)L)+L+L^2-L\right]
\\
= |I| \ln((k+1)L)+|I| L+|I|  L(L-1).
\end{multline*}

Denote $V=\{i\in M\mid \upsilon'_i<\upsilon_i\}.$ Clearly, $I\subseteq V.$
Next we prove the following three inequalities:
\begin{equation}
|I|\ln((k+1)L)\geqslant\sum_{i\in I}\ln \upsilon_i; 
\label{T8}
\end{equation}
\begin{equation}
|I|L>\sum_{i\in I} k
\ln \upsilon_i-
\sum_{i\in I} k\ln \upsilon'_i;
\label{T9}
\end{equation}
\begin{equation}
|I|L(L-1)>\sum_{i\in V\setminus I} n_i\ln \upsilon_i-
\sum_{i\in V\setminus I} n'_i
\ln \upsilon'_i.
\label{I}
\end{equation}

If \eqref{T8}, \eqref{T9} and \eqref{I} are proven, then 
\begin{multline*}
f(\alpha)-f(\alpha')>\sum_{i\in  I}\ln \upsilon_i+\sum_{i\in  I}k\ln \upsilon_i-\sum_{i\in  I}k\ln \upsilon'_i+\sum_{i\in V\setminus I}n_i\ln \upsilon_i-\sum_{i\in V\setminus I}n'_i\ln \upsilon'_i
\\
=\sum_{i\in  I}(k+1)\ln \upsilon_i+\sum_{i\in V\setminus I}n_i\ln \upsilon_i-\left(\sum_{i\in I}n'_i\ln \upsilon'_i+\sum_{i\in V\setminus I}n'_i\ln \upsilon'_i \right)
\\
=\sum_{i=1}^m n_i\ln \upsilon_i-\sum_{i=1}^m n'_i\ln \upsilon'_i=g(\alpha)-g(\alpha'),
\end{multline*}
since for $i\in I$, $n_i=k+1$, $n'_i=k$, and for $i\in M\setminus V$, $n'_i=n_i$ and $\upsilon'_i=\upsilon_i$. That proves \eqref{eq:ind_step}.
\begin{center}
Proof of \eqref{T8}
\end{center}

By \eqref{eq:est_epsilon}, $\upsilon_i\leqslant (k+1)L$ and \[\sum_{i\in I} \ln (\upsilon_i)\leqslant
\sum_{i\in I} \ln\left((k+1)L \right)=|I|\ln\left((k+1)L \right).\]
\newpage
\begin{center}
Proof of \eqref{T9}
\end{center} 
\[\text{Since }\upsilon_i-\upsilon'_i=\sum_{\{j\in M:T_j\text{ is linked to }T_i\}}\left(n_j-n'_j \right) \leqslant L,\text{ we have }\upsilon_i\leqslant\upsilon'_i+L.\] 
For $i\in I$, $\upsilon'_i\geqslant n'_i=k$, and using \eqref{eq:logarithm} we get:
\[k\ln \left( \frac{\upsilon_i}{\upsilon'_i}\right) \leqslant k\ln \left( \frac{\upsilon'_i+L}{\upsilon'_i}\right)=k\ln \left(1+ \frac{L}{\upsilon'_i}\right)<k\frac{L}{\upsilon'_i}\leqslant k\frac{L}{k}=L;\] that is 
$k\ln \left( \frac{\upsilon_i}{\upsilon'_i}\right)<L.$ So

\[\sum_{i\in I} k
\ln \upsilon_i-
\sum_{i\in I} k\ln \upsilon'_i=\sum_{i\in I}k\ln \left( \frac{\upsilon_i}{\upsilon'_i}\right)<\sum_{i\in I}L=|I| L.\]

\begin{center}
Proof of \eqref{I}
\end{center} 

For $i\in V\setminus I$, $n'_i=n_i$. By \eqref{eq:logarithm} and since $\upsilon_i\leqslant\upsilon'_i+L$, we have:
\begin{multline*}
\sum_{i\in V\setminus I} n_i\ln \upsilon_i-
\sum_{i\in V\setminus I} n'_i
\ln \upsilon'_i=
\sum_{i\in V\setminus I} n_i\ln \upsilon_i-
\sum_{i\in V\setminus I} n_i
\ln \upsilon'_i
\\
=\sum_{i\in V\setminus I} n_i\ln \left( \frac{\upsilon_i}{\upsilon'_i}\right)
\leqslant
\sum_{i\in V\setminus I} n_i\ln \left( \frac{\upsilon'_i+L}{\upsilon'_i}\right)
=\sum_{i\in V\setminus I} n_i\ln \left(1+\frac{L}{\upsilon'_i}\right)
\\
<\sum_{i\in V\setminus I} n_i\frac{L}{\upsilon'_i}\leqslant
\sum_{i\in V\setminus I} n_i\frac{L}{n_i}=L| V\setminus I|
\text{, since }\upsilon'_i\geqslant n'_i=n_i\text{ for }i\in V\setminus I.
\end{multline*}
It remains to prove that $|V\setminus I|\leqslant |I|(L-1)$, which is equivalent to:
\begin{equation}
| V|\leqslant |I|L.
\label{eq:cardinality}
\end{equation}
\begin{center}
Proof of \eqref{eq:cardinality}
\end{center} 

Denote $V_i=\{j\in V\mid T_i\text{ and }T_j\text{ are linked}\}$. Then $| V_i|\leqslant L$.

If $j\notin \bigcup_{i\in I}V_i$, then no vertices $T_i$ with $i\in I$ are linked to $T_j$ and $\upsilon'_j=\upsilon_j$, so $j\notin V$. Therefore \[V\subseteq \bigcup_{i\in I}V_i\text{ and }|V|\leqslant\sum_{i\in  I}| V_i|\leqslant\sum_{i\in I}L=|I| L.\]

This completes the proof of Theorem \ref{theorem:estimation}. 
\end{proof}

\subsection{Application of the Estimation Theorem to semi-invariants}
In this subsection two subsets $A,B$ of $\mathbb{Z}^\nu$ are said to be \textit{linked} iff $A\cap B\neq \varnothing$ and $\mathfrak{A}$ is a set of finite non-empty subsets of $\mathbb{Z}^\nu$. We assume there exists a constant $L$ such that each $A\in\mathfrak{A}$ is linked to at most $L$ elements of $\mathfrak{A}$.

\begin{lemma}
Denote $C_{1}=3Le^{L^2+1}$. For any family $\alpha$ of elements of $\mathfrak{A}$ with the associated sequence
$\widehat{\alpha}=(T_1,T_2,\ldots,T_n)$:
\[\Big|\langle \prod_{t\in T_1}Q_t,
\prod_{t\in T_2}Q_t,\ldots,\prod_{t\in T_n}Q_t\rangle_0\Big|\leqslant 
(C_{1})^{|\alpha|}\alpha!\]
\label{lemma:estimates_1}
\end{lemma}
\begin{proof}
For brevity we denote the left hand side of this inequality by $H$. We can write the family $\alpha$ in the form $\alpha=\{(\overline{T}_1,n_1), \ldots,(\overline{T}_q,n_q)\}$. Then $n_1+n_2+\ldots+n_q=n=|\alpha|$ and $\alpha!=n_1!\cdot n_2!\cdot\ldots\cdot n_q!$

Theorem 1 on page 69 of \cite{MM91} implies that:
\begin{multline}
H\leqslant  \dfrac{3}{2}
\prod_{j=1}^n3\upsilon(T_j)
=\dfrac{3}{2}
\prod_{i=1}^q\left(3\upsilon(\overline{T}_i)\right)^{n_i}
=\dfrac{3}{2}
\prod_{i=1}^q3^{n_i}
\prod_{i=1}^q\left(\upsilon(\overline{T}_i)\right)^{n_i}
\\
=\dfrac{3}{2}
3^{|\alpha|}
\prod_{i=1}^q\left(\upsilon(\overline{T}_i)\right)^{n_i}.
\label{eq:Malyshev_0}
\end{multline}

By the Estimation Theorem (Theorem \ref{theorem:estimation}):
\begin{multline*}
\ln\left(\prod_{i=1}^q\left(\upsilon(\overline{T}_i)\right)^{n_i} \right) 
=
\sum_{i=1}^q n_i\ln \upsilon(\overline{T}_i)
<
C_L\sum_{i=1}^q n_i +\sum_{i=1}^q n_i \ln n_i
\\
=C_L|\alpha| +\sum_{i=1}^q n_i \ln n_i,
\end{multline*}
where $C_L=\ln L+L^2$. So
\[\prod_{i=1}^q\left(\upsilon(\overline{T}_i)\right)^{n_i} 
\leqslant e^{C_L|\alpha|}\prod_{i=1}^q e^{n_i\ln n_i}
= e^{C_L|\alpha|}\prod_{i=1}^q n_i^{n_i}.\]

By Stirling's formula, for any natural number $k$: $k^k=\dfrac{k!e^k}{\theta_k\sqrt{2\pi k}}$, where 
\\
$1<\theta_k<e$; so $k^k\leqslant \dfrac{1}{2} k!e^k.$ Then by \eqref{eq:Malyshev_0}:
\begin{multline*}
H\leqslant  \dfrac{3}{2}3^{|\alpha|}e^{C_L|\alpha|}\prod_{i=1}^q n_i^{n_i}\leqslant (3e^{C_L})^{|\alpha|}\prod_{i=1}^q n_i!e^{n_i}
= (3e^{\ln L+L^2})^{|\alpha|}e^{|\alpha|} \alpha!
\\
=(3Le^{L^2}e)^{|\alpha|}\alpha!=
(3Le^{L^2+1})^{|\alpha|}\alpha!
=(C_1)^{|\alpha|}\alpha!
\end{multline*}
\end{proof}

\section{Proof of the central limit theorem for Ising model}
\label{section:proof}

\subsection{Proof of Theorem 2.1}
\label{subsection_proof1}

In this subsection we prove a series of lemmas about estimates and limits of semi-invariants, and we use these lemmas in  subsections 5.2 and 5.3 for direct proof of the main theorem.

The following lemma is mentioned by several authors without a proof or with a complicated proof. Here we provide a short, simple proof giving an explicit value for the estimation constant. 

\begin{lemma}
Denote $C_2=4\nu^2$. Fix a sequence $b=(t_1,\ldots,t_m)$ of points in $\mathbb{Z}^\nu$, $m\geqslant 1$, and a natural number $n\geqslant 1$. The number of families $\gamma\in R^*$ such that $|\gamma|=n$ and $\gamma$ connects $b$, is not greater than $(C_2)^n$.
\label{lemma:Eulirean}
\end{lemma}
\begin{proof}
Consider the associated sequence $\widehat{\gamma}=(A_1,\ldots,A_n)$ and the associated graph $G$ of $b$ and $\gamma$. A new graph $G'$ is obtained from $G$ by adding for each $i=1,\ldots,n$, an extra edge $A'_i$ that has the same ends as $A_i$. Then each vertex in $G'$ has an even degree and hence $G'$ has an Eulirean path, that is a closed path which includes every edge of the graph exactly once; the length of such a path is $2n$. 

Therefore the number of the families with $|\gamma|=n$ that connect $b$, is not greater than the number of paths with $2n$ steps through $t_1,\ldots,t_m$ going along edges of the lattice $\mathbb{Z}^\nu$. There are at most $2\nu$ directions at each step. Therefore the number of such paths is not greater than $(2\nu)^{2n}=(C_2)^n$ for $C_2=(2\nu)^2$. 
\end{proof}

\begin{lemma}
Suppose $A_1,A_2,\ldots,A_n\in R$ and $t\in \mathbb{Z}^\nu$. Then
\[\langle Q_{t},
\Phi_{A_1},\ldots, \Phi_{A_n}\rangle_0=0.\]
\label{lemma:capital_phi_1}
\end{lemma}
\begin{proof}
Denote $M=\{t,A_1,A_2,\ldots,A_n\}$ the set of indices, where repeated elements are counted separately and each $A_j=\{r_j,s_j\}\subset\mathbb{Z}^\nu$. By the definition of semi-invariants, 
\[\langle Q_{t},
\Phi_{A_1},\ldots, \Phi_{A_n}\rangle_0=\sum_\alpha (-1)^{k-1} (k-1)!\langle Q_{t}\cdot\Phi_{S_1}\rangle_0\langle \Phi_{S_2}\rangle_0\ldots \langle\Phi_{S_k}\rangle_0,\]
where the sum is taken over all partitions  
$\alpha=\lbrace \{t\}\cup S_1,S_2,\ldots,S_k\rbrace $ 
of $M$; each $S_i\subseteq \{A_1,A_2,\ldots,A_n\}$; without loss of generality we can assume that $t$ belongs to the first set in each partition. We will show that each addend in this sum equals 0. 

In partition $\alpha$, $S_1$ has the form: $S_1=\{A_{j_1},\ldots,A_{j_q}\},q\geqslant0$. The addend corresponding to $\alpha$ contains this multiplier:
\begin{multline*}
\langle Q_{t}\cdot\Phi_{S_1}\rangle_0=\langle Q_{t}\cdot\cdot
\Phi_{A_{j_1}}\cdot\ldots\cdot \Phi_{A_{j_q}}\rangle_0
\\
=\langle Q_{t}\cdot (Q_{r_{j_1}}\cdot Q_{s_{j_1}})\cdot\ldots\cdot
(Q_{r_{j_l}}\cdot Q_{s_{j_l}})
\rangle_0
=0
\end{multline*}
by Lemma \ref{lemma:zero_lambda}.6) because the number of multipliers in the last product equals $2q+1$, which is odd. This completes the proof of the lemma.
\end{proof}

\begin{lemma}
Suppose $b$ is a sequence of points in $\mathbb{Z}^\nu$, $\gamma\in R^*$ and $\gamma$ does not connect $b$. Then $\langle Q^,_b,\Phi^\backprime_\gamma\rangle_0=0.$
\label{lemma:capital_phi_2}
\end{lemma}
\begin{proof}
Consider the associated graph $G$ of $b$ and $\gamma$. Since $\gamma$ does not connect $b$, we have $G=G_1\cup G_2$, where $G_1$ and $G_2$ are disjoint graphs. Without loss of generality we can write:
\[b=(t_1,\ldots,t_k,s_1,\ldots,s_l)\text{ and }\widehat{\gamma}=(A_1,\ldots,A_m,B_1,\ldots,B_n),\] where $G_1$ corresponds to $(t_1,\ldots,t_k),(A_1,\ldots,A_m)$ and $G_2$ corresponds to \\$(s_1,\ldots,s_l),(B_1,\ldots,B_n)$. 

By Lemma \ref{lemma:zero_lambda}.2), the random vectors $(Q_{t_1},\ldots,Q_{t_k},\Phi_{A_1},\ldots, \Phi_{A_m})$ and \\$(Q_{s_1},\ldots, Q_{s_l},\Phi_{B_1},\ldots, \Phi_{B_n})$ are independent. So by Lemma \ref{lemma:sem}.2):
\[\langle Q^,_b,\Phi^\backprime_\gamma\rangle_0=\langle Q_{t_1},\ldots,Q_{t_k},
\Phi_{A_1},\ldots,\Phi_{A_{m}},
Q_{s_1},\ldots,Q_{s_l},\Phi_{B_1},\ldots, \Phi_{B_n}\rangle_0=0.\]
\end{proof}

\begin{lemma}
Suppose $f:\mathbb{Z}^\nu\rightarrow\mathbb{Z}^\nu$ is a bijection preserving distances. Suppose $t_1,\ldots,t_m\in\mathbb{Z}^\nu$ and $A_1,\ldots,A_n\in R$, where each $A_j=\{r_j,q_j\}$. 
\medskip

Denote $s_i=f(t_i)$, $i=1,\ldots,m$, and $B_j=\{f(r_j),f(q_j)\}$, $j=1,\ldots,n$. 

Then the random vectors $(Q_{t_1},\ldots,Q_{t_m},\Phi_{A_1},\ldots,\Phi_{A_n})$ and 
\\
$(Q_{s_1},\ldots,Q_{s_m},\Phi_{B_1},\ldots,\Phi_{B_n})$ have the same distribution with respect to measure $P_{0,N}$.
\label{lemma:translation_2}
\end{lemma}
\begin{proof}
Since $f$ preserves distances, we have $\rho(f(r_j),f(q_j))=1$ for any $j=1,\ldots,n$. The lemma follows from Lemma \ref{lemma:translation_1} because for any $A=\{r,q\}$ and $a\in \{-1,1\}$ we can write: 
\[P(\Phi_{A}=a)=P(Q_r\cdot Q_q=a)
=P(Q_r=1,Q_q=a)+P(Q_r=-1,Q_q=-a).\]
\end{proof}

\begin{lemma}
Denote $L=4\nu+1$ and  $C_3=3Le^{L^2+1}$. For any sequence $b=(t_1,\ldots,t_m)$ of points in $\mathbb{Z}^\nu$ and any family $\gamma\in R^*$: 
\[|\langle Q^,_b,\Phi^\backprime_\gamma\rangle_0|\leqslant 
(C_3)^{m+|\gamma|}\;m!\;\gamma!\]
\label{lemma:estimates}
\end{lemma}
\begin{proof}
We define a set $\mathfrak{A}=R\cup \left\lbrace \{t\}\mid t\in\mathbb{Z}^\nu\right\rbrace$. So $\mathfrak{A}$ is a set of finite non-empty subsets of $\mathbb{Z}^\nu$.
Any element of the form $\{t\}$ is linked to itself and to $2\nu$ elements of the form $\{t,r\}$, so the total is $2\nu+1$. Any element\ of the form $\{r,s\}$ is linked to elements $\{r\}$, $\{s\}$, $2\nu$ elements of the form $\{r,t\}$ and $2\nu$ elements of the form $\{s,t\}$, so the total is $4\nu+1$ (because the element $\{r,s\}$ is counted twice). Thus, each element of $\mathfrak{A}$ is linked to at most $L$ elements and we can apply Lemma \ref{lemma:estimates_1}.

Consider a sequence $b=(t_1,t_2,\ldots,t_m)$ of points in $\mathbb{Z}^\nu$ and $\gamma\in R^*$. Let $\overline{t}_1,\overline{t}_2,\ldots,\overline{t}_q$ be the points $t_1,t_2,\ldots,t_m$ written without repetitions. 
\medskip

Denote $\beta=\left\lbrace \left(\{\overline{t}_1\},n_1\right),\ldots,
\left(\{\overline{t}_q\},n_q\right) \right\rbrace$. Then $n_1+n_2+\ldots+n_q=m$ and 
\medskip
\\
$n_1!n_2!\ldots n_q!\leqslant(n_1+n_2+\ldots+n_q)!=m!$. 
Denote $\alpha=\beta\cup\gamma$. 
Then $\widehat{\alpha}$ has the 
\medskip
\\
form $\widehat{\alpha}=\left(\{t_1\},\ldots,\{t_m\},A_1,\ldots A_r\right)$, where $\widehat{\gamma}=(A_1,\ldots A_r)$. 
\medskip

Clearly, $|\alpha|=m+|\gamma|$ and $\alpha!=n_1!n_2!\ldots n_q!\gamma!\leqslant m!\gamma!$.
\medskip

For any $i=1,2,\ldots,r$, $\prod_{t\in A_i}Q_t=\Phi_{A_i}$. So by Lemma \ref{lemma:estimates_1}, 
\begin{multline*}
|\langle Q^,_b,\Phi^\backprime_\gamma\rangle_0|
=\Big|\langle \prod_{t\in \{t_1\}}Q_t,\ldots,\prod_{t\in \{t_m\}}Q_t,\prod_{t\in A_1}Q_t,\ldots,\prod_{t\in A_r}Q_t
\rangle_0\Big|
\leqslant (C_3)^{|\alpha|}\alpha!
\\
\leqslant (C_3)^{m+|\gamma|}\;m!\;\gamma!
\end{multline*}
This completes the proof of the lemma.
\end{proof}

In subsection \ref{subsection:Ising model} we introduced the set of all pairs of neighbouring nodes in the cube $\Lambda_N$:
$R_N=\left\lbrace \{s,t\}\mid s,t\in\Lambda_N \;\&\; \rho(s,t)=1\right\rbrace$.

Denote $R^*_N=\{\gamma\mid\gamma$ is a family of elements of $R_N\}$.
The following lemma describes a connection between the semi-invariants with respect to the measures $P_{\lambda,N} $ and $P_{0,N}$. 
\begin{lemma}
Denote $C_\nu=\dfrac{1}{2C_2C_3}$, where $C_2$ and $C_3$ are the positive constants from Lemmas \ref{lemma:Eulirean} and \ref{lemma:estimates}, respectively, depending only on $\nu$. Suppose $|\lambda|<C_\nu$. Then for any $K\geqslant 1$, sequence $b=(t_1,t_2,\ldots,t_m)$ of points in $\Lambda_{K}$ and $N\geqslant K$: 
\begin{equation}
\langle Q^,_b
\rangle_{\lambda,N}
=\sum_{n=0}^{\infty} \lambda^n
\sum_{\{\gamma\in R^*_N:|\gamma|=n\}} \dfrac{1}{\gamma !}
\langle Q^,_b,\Phi^\backprime_\gamma\rangle_0.
\label{eq:Taylor_series}
\end{equation}
The inner sum can be taken over the families $\gamma$ that connect $b$.

The series in the formula \eqref{eq:Taylor_series} converges absolutely and uniformly for all $N\geqslant K$. The semi-invariants on the right-hand side are with respect to measure $P_{0,N}$ and do not depend on $N$, due to Lemma \ref{lemma:zero_lambda}.3).
\label{lemma:series_N}
\end{lemma}
\begin{proof}
Suppose $|\lambda|<C_\nu$. Fix a sequence $b=(t_1,t_2,\ldots,t_m)$ of points in $\Lambda_{K}$ and $N\geqslant K$. Semi-invariants with respect to Gibbs measure can be expanded in Taylor series; the proof was given in \cite{MM91}, pg. 34. In our case the Taylor series has the form: 
\begin{equation}
\langle Q^,_b\rangle_{\lambda,N}=\langle Q_{t_1},Q_{t_2},\ldots
,Q_{t_m} \rangle_{\lambda,N}
=\sum_{n=0}^{\infty} \frac{\lambda^n}{n!} \langle Q_{t_1},Q_{t_2},\ldots,Q_{t_m},
\underbrace{W,\ldots,W}_{n\text{ times}}\rangle_0,
\label{eq:Malyshev_2}
\end{equation}
where $W=\sum_{A\in R_N}\Phi_A$.
Let us consider the Taylor coefficients:
\begin{multline}
a_{N,n}=\frac{1}{n!} \langle Q_{t_1},Q_{t_2},\ldots,Q_{t_m},
\underbrace{W,\ldots,W}_{n\text{ times}}\rangle_0
\\
=\frac{1}{n!} 
\sum_{A_1\in R_N} 
\sum_{A_2\in R_N}\ldots\sum_{A_n\in R_N}
 \langle Q_{t_1},Q_{t_2},\ldots
,Q_{t_m},
\Phi_{A_1},\Phi_{A_2},\ldots, \Phi_{A_n}\rangle_0. 
\label{eq:Gibm1}
\end{multline}

Then the Taylor series for the semi-invariant has the form:
\begin{equation}
\langle Q^,_b\rangle_{\lambda,N}=\sum_{n=0}^{\infty}\lambda^n a_{N,n}.
\label{eq:Taylor}
\end{equation}

Consider any $A_1, A_2,\ldots,A_n\in R_N$. This sequence can contain repeating elements. Denote the corresponding family $\gamma=\{(\overline{A_1},n_1), \ldots,(\overline{A_q},n_q)\}$. There are $\dfrac{n!}{\gamma!}$ ordered sequences $(A_1, A_2,\ldots,A_n)$ that reduce to the same family $\gamma$ of length $n$. Therefore \eqref{eq:Gibm1} can be written as:
\[a_{N,n}
=\frac{1}{n!}\sum_{\{\gamma\in R^*_N:|\gamma|=n\}} \dfrac{n!}{\gamma !}\langle Q^,_b,\Phi^\backprime_\gamma\rangle_0
=\sum_{\{\gamma\in R^*_N:|\gamma|=n\}} \dfrac{1}{\gamma !}
\langle Q^,_b,\Phi^\backprime_\gamma\rangle_0.\]
In this sum we can take only the families $\gamma$ that connect $b$ because for others the corresponding addends equal 0 by Lemma \ref{lemma:capital_phi_2}. By Lemma \ref{lemma:Eulirean}, the number of such families $\gamma$ is not greater than $(C_2)^n$.

By Lemma \ref{lemma:estimates}, for each $\gamma$: $|\langle Q^,_b,\Phi^\backprime_\gamma\rangle_0|\leqslant 
(C_3)^{m+|\gamma|}\;m!\;\gamma!$ If $|\gamma|=n$, we have: 
\begin{multline*}
|\lambda^n a_{N,n}|\leqslant 
|\lambda|^n \sum_{\{\gamma\in R^*_N:|\gamma|=n\}} \dfrac{1}{\gamma !}\Big|\langle Q^,_b,\Phi^\backprime_\gamma\rangle_0\Big|
\leqslant |\lambda|^n(C_2)^n(C_3)^{m+n}\;m!
\\
=(C_3)^{m}\;m!\;|\lambda C_2C_3|^n.
\end{multline*}

Since $|\lambda|<C_\nu$, then each $|\lambda^n a_{N,n}|\leqslant (C_3)^{m}\;m!\;2^{-n}$. So the series \eqref{eq:Taylor} converges absolutely and uniformly for all $N\geqslant K$. This completes the proof of the lemma.
\end{proof} 

\begin{lemma}
Suppose $|\lambda|<C_\nu$, where $C_\nu$ is the constant from Lemma \ref{lemma:series_N}. Then for any sequence $b$ of points in $\mathbb{Z}^\nu$ the following limit exists and
\[\lim_{N\rightarrow\infty}\langle Q^,_b\rangle_{\lambda,N}
=\sum_{n=0}^{\infty} \lambda^n
\sum_{\{\gamma\in R^*:|\gamma|=n\}} \dfrac{1}{\gamma !}
\langle Q^,_b,\Phi^\backprime_\gamma\rangle_0.\]

On the right-hand side the series converges and the inner sum can be taken over the families $\gamma$ that connect $b$.
\label{lemma:semi_limit}
\end{lemma}
\begin{proof}
Suppose $|\lambda|<C_\nu$. Let us fix a sequence $b=(t_1,\ldots,t_m)$ of points in $\mathbb{Z}^\nu$. There is a  sufficiently big $K$ such that $t_1,\ldots,t_m\in\Lambda_{K}$. By Lemma \ref{lemma:series_N} for any $N\geqslant K$: 
\[\langle Q^,_b
\rangle_{\lambda,N}
=\sum_{n=0}^{\infty} \lambda^n
\sum_{\{\gamma\in R^*_N:|\gamma|=n\}} \dfrac{1}{\gamma !}
\langle Q^,_b,\Phi^\backprime_\gamma\rangle_0\]
and the series conversges uniformly for all $N\geqslant K$. Therefore
\begin{multline*}
\lim_{N\rightarrow\infty}\langle Q^,_b\rangle_{\lambda,N}
=\sum_{n=0}^{\infty} \lambda^n
\lim_{N\rightarrow\infty}
\left(\sum_{\{\gamma\in R^*_N:|\gamma|=n\}} \dfrac{1}{\gamma !}
\langle Q^,_b,\Phi^\backprime_\gamma\rangle_0\right)
\\
=\sum_{n=0}^{\infty} \lambda^n
\sum_{\{\gamma\in R^*:|\gamma|=n\}} \dfrac{1}{\gamma !}
\langle Q^,_b,\Phi^\backprime_\gamma\rangle_0\text{ and the series converges}.
\end{multline*}
\end{proof}
 
\begin{center}
\textbf{Proof of Theorem \ref{theorem:therm_limit}}
\end{center}

The constant $C_\nu$ was defined in Lemma \ref{lemma:series_N}.
Part 1 follows from Lemma \ref{lemma:semi_limit}.
Part 2 follows from part 1 and Lemma \ref{lemma:sem}.3).

Theorem \ref{theorem:limits} follows from Theorem \ref{theorem:therm_limit} as shown in Section 2.

\begin{corollary}
Suppose $|\lambda|<C_\nu$, where $C_\nu$ is the constant from Lemma \ref{lemma:series_N}. Then the following hold. 

1. For any sequence $b=(t_1,t_2,\ldots,t_m)$ of points in $\mathbb{Z}^\nu$:
\[\langle Q^,_b
\rangle_{\lambda}
=\sum_{n=0}^{\infty} \lambda^n
\sum_{\{\gamma\in R^*:|\gamma|=n\}} \dfrac{1}{\gamma !}
\langle Q^,_b,\Phi^\backprime_\gamma\rangle_0.\]
On the right-hand side the series converges and the inner sum can be taken over the families $\gamma$ that connect $b$.
\medskip

2. For any $t_1,t_2,\ldots,t_m\in\mathbb{Z}^\nu$: $\quad
\langle Q_{t_1},Q_{t_2},\ldots
,Q_{t_m}\rangle_{\lambda}=$
\[=\sum_{n=0}^{\infty} \frac{\lambda^n}{n!}\sum_{A_1\in R}\sum_{A_2\in R}\ldots\sum_{A_n\in R}\langle Q_{t_1},Q_{t_2},\ldots,Q_{t_m},\Phi_{A_1},\Phi_{A_2},\ldots, \Phi_{A_n}\rangle_0.\]
The series on the right-hand side converges.
\medskip

3. For any $t\in\mathbb{Z}^\nu$, 
$\langle Q_t\rangle_\lambda=0$ and $\langle Q_t,Q_t\rangle_\lambda=1$.

4. Suppose $f:\mathbb{Z}^\nu\rightarrow\mathbb{Z}^\nu$ is a bijection preserving distances.  Then for any 
\\$t_1,\ldots,t_m\in\mathbb{Z}^\nu$:
\[\langle Q_{t_1},\ldots,Q_{t_m}\rangle_\lambda=\langle Q_{f(t_1}),\ldots,Q_{f(t_m)}\rangle_\lambda.\]

5. Fix $a\in \mathbb{Z}^\nu$ and  define $g:\mathbb{Z}^\nu\rightarrow\mathbb{Z}^\nu$ by the following: $g(t)=t-a.$ 

Suppose a renormalization group with parameters $k$ and  $\alpha$  transforms the random field $\{Q_t\mid t\in\mathbb{Z}^\nu\}$ into a random field $\lbrace   Y^{(k)}_\tau \mid\tau\in Z^\nu \rbrace$. 
Then for any $\tau_1,\ldots,\tau_m\in\mathbb{Z}^\nu$:
\[\langle Y^{(k)}_{g(\tau_1)},\ldots,Y^{(k)}_{g(\tau_m)}\rangle_\lambda=\langle Y^{(k)}_{\tau_1},\ldots,Y^{(k)}_{\tau_m}\rangle_\lambda.\]

\label{corollary}
\end{corollary}
\begin{proof}
1. It follows from Theorem \ref{theorem:limits}.2) and 
Lemma \ref{lemma:semi_limit}.

2. It is proven by re-arranging the sum in part 1, similarly to the proof of Lemma \ref{lemma:series_N}.

3. By part 2: \[\quad
\langle Q_t\rangle_{\lambda}
=\sum_{n=0}^{\infty} \frac{\lambda^n}{n!} 
\left\lbrace \sum_{A_1\in R} 
\sum_{A_2\in R}\ldots\sum_{A_n\in R}\langle Q_t,
\Phi_{A_1},\Phi_{A_2},\ldots, \Phi_{A_n}\rangle_0 \right\rbrace=0,\]
since each addend equals 0 by Lemma \ref{lemma:capital_phi_1}.
\smallskip

Since $Q_t^2=1$, then $\langle Q_t,Q_t\rangle_\lambda=\langle Q_t^2\rangle_\lambda-(\langle Q_t\rangle_\lambda)^2=1-0=1$.
\smallskip

4. Define a transformation $F$ by: $F(\{r,q\})=\{f(r),f(q)\}$.  Denote $s_i=f(t_i)$, $i=1,\ldots,m$. By part 2 and Lemma \ref{lemma:translation_2}, 
\begin{multline*}
\langle Q_{t_1},\ldots,Q_{t_m}\rangle_\lambda
=\sum_{n=0}^{\infty} \frac{\lambda^n}{n!} 
\left\lbrace \sum_{A_1\in R} 
\ldots\sum_{A_n\in R}\langle Q_{t_1},\ldots,Q_{t_m},
\Phi_{A_1},\ldots, \Phi_{A_n}\rangle_0 \right\rbrace
\\
=\sum_{n=0}^{\infty} \frac{\lambda^n}{n!} 
\left\lbrace \sum_{A_1\in R} 
\ldots\sum_{A_n\in R}\langle Q_{s_1},\ldots,Q_{s_m},
\Phi_{F(A_1)},\ldots, \Phi_{F(A_n)}\rangle_0 \right\rbrace
\\
=\sum_{n=0}^{\infty} \frac{\lambda^n}{n!} 
\left\lbrace \sum_{B_1\in R} 
\ldots\sum_{B_n\in R} \langle Q_{s_1},\ldots,Q_{s_m},
\Phi_{B_1},\ldots, \Phi_{B_n}\rangle_0 \right\rbrace
\\
=\langle Q_{s_1},\ldots,Q_{s_m}\rangle_\lambda
\end{multline*}
because $F$ is a bijection on $R$.
\medskip

5. Fix $\tau_1,\ldots,\tau_m\in\mathbb{Z}^\nu$. Define $f:\mathbb{Z}^\nu\rightarrow\mathbb{Z}^\nu$ by the following: $f(t)=t-ka.$ Then $f$ is a bijection preserving distances. Clearly, for any $i=1,\ldots,m$:
\begin{equation}
G_k^{-1}(g(\tau_i))=\left\lbrace f(t)\mid t\in G_k^{-1}(\tau_i)\right\rbrace.
\label{eq:cubes}
\end{equation}

By the definition of renormalization-group, we have:
\begin{multline*}
\langle Y^{(k)}_{g(\tau_1)},\ldots,Y^{(k)}_{g(\tau_m)}\rangle_\lambda 
\\
=(k^{-\alpha/2})^m\sum_{s_1 \in G_k^{-1}(g(\tau_1))}\ldots\sum_{s_m \in G_k^{-1}(g(\tau_m))}\langle Q_{s_1},\ldots,Q_{s_m}\rangle_\lambda
\\
=[\text{by \eqref{eq:cubes}}]
=k^{-\alpha m/2}\sum_{t_1 \in G_k^{-1}(\tau_1)}\ldots\sum_{t_m \in G_k^{-1}(\tau_m)}\langle Q_{f(t_1)},\ldots,Q_{f(t_m)}\rangle_\lambda
\\
=[\text{by part 4}]
=k^{-\alpha m/2}\sum_{t_1 \in G_k^{-1}(\tau_1)}\ldots\sum_{t_m \in G_k^{-1}(\tau_m)}\langle Q_{t_1},\ldots,Q_{t_m}\rangle_\lambda
\\
=\langle Y^{(k)}_{\tau_1},\ldots,Y^{(k)}_{\tau_m}\rangle_\lambda.
\end{multline*}
\end{proof}

\subsection{Finding the limiting variances}

\begin{lemma}
For any $0<x<\dfrac{1}{2}$ the following hold.

1. For any $l=0,1,2,\ldots$,
\[\text{the series }\sum_{n=l+1}^\infty(n-l)x^n\text{ converges and }
\sum_{n=l+1}^\infty(n-l)x^n=\dfrac{x^{l+1}}{(1-x)^2}.\]

2. For any $m=1, 2,\ldots$,
\[\text{the series }\sum_{n=0}^\infty(n+1)^{m-1}x^n \text{  converges and }
\sum_{n=0}^\infty(n+1)^{m-1}x^n\leqslant\dfrac{m!}{(1-x)^{m+1}}.\]
\label{lemma:number_series}
\end{lemma}
\begin{proof}
\begin{multline*}
1.\;\sum_{n=l+1}^\infty(n-l)x^n=
[\text{substitution }k=n-l-1]=x^{l+1}\sum_{k=0}^\infty(k+1)x^k
\\
=x^{l+1}\sum_{k=0}^\infty(x^{k+1})'=x^{l+1}\left( \sum_{k=0}^\infty x^{k+1}+1\right)' 
=x^{l+1}\left(\dfrac{1}{1-x}\right)'=x^{l+1}\dfrac{1}{(1-x)^2}.
\end{multline*}

2. For each $n$, $(n+1)^{m-1}x^n\leqslant(n+1)^{m}x^n\leqslant(n+1)(n+2)\ldots(n+m)x^n=
\\=\left( x^{n+m}\right)^{(m)}.$ Since each of the series $\sum_{n=0}^\infty x^n$ and $\sum_{i=0}^\infty \left( x^{i}\right)^{(m)}$ absolutely and uniformly converges on $\left[0,\dfrac{1}{2}\right]$, then
\[\sum_{n=0}^\infty \left(x^{n+m}\right)^{(m)}=\sum_{i=0}^\infty\left(x^{i}\right)^{(m)}=\left(\sum_{i=0}^\infty x^{i}\right)^{(m)} =\left(\dfrac{1}{1-x}\right)^{(m)},\]
because for $i<m$, $\left(x^{i}\right)^{(m)}=0$. It is easily proven by induction on $m$ that:
\[\left(\dfrac{1}{1-x}\right)^{(m)}=\dfrac{m!}{(1-x)^{m+1}}.\]

So the series $\sum_{n=0}^\infty(n+1)^{m-1}x^n$ converges and 
\[\sum_{n=0}^\infty(n+1)^{m-1}x^n\leqslant\dfrac{m!}{(1-x)^{m+1}}.\]
\end{proof}

In the following theorem we derive explicit expressions for the limiting variances of $Y^{(k)}_{\tau}$. This theorem is interesting by itself and also becomes a part of the direct proof of the Main Theorem in subsection \ref{proof_main}. 

\begin{theorem}
Denote $C=\min\left\lbrace \dfrac{1}{2C_2C_3},\dfrac{1}{8C_2(C_3)^3}
\right\rbrace$, where $C_2$ and $C_3$ are the positive constants from Lemmas \ref{lemma:Eulirean} and \ref{lemma:estimates}, respectively, depending only on $\nu$. 

Consider the random field   $\lbrace   Y^{(k)}_\tau \mid\tau\in \mathbb{Z}^\nu \rbrace$ from the Main Theorem (Theorem \ref{theorem:main}) and $\alpha=\nu$. Then for any $|\lambda|<C$ and $\tau\in\mathbb{Z}^\nu$: 
\[\lim_{k\rightarrow\infty}\langle Y^{(k)}_{\tau},Y^{(k)}_{\tau}\rangle_\lambda=1+\sum_{n=1}^{\infty}
\lambda^nV_n,\]
where the series on the right-hand side converges and
\[\text{each }V_n=\sum_{\substack{\gamma\in R^*,
\\
|\gamma|=n,
\\
\gamma\text{ connects }(\bar{0})
}} 
\sum_{\substack{t\in \tilde{\gamma},
\\
\gamma\text{ connects }(\bar{0},t),
\\
t\neq \bar{0}
}}
\dfrac{1}{\gamma!}
\langle Q_{\bar{0}},Q_t,\Phi^\backprime_\gamma\rangle_0.\]

The limiting variances are positive and do not depend on $\tau$.
\label{theorem:variance}
\end{theorem}
\begin{proof}

Suppose $|\lambda|<C$. Then $|\lambda|<C_\nu$, where $C_\nu=\dfrac{1}{2C_2C_3}$ is the constant from Lemma \ref{lemma:series_N}.

Fix $\tau\in \mathbb{Z}^\nu$. Define $g:\mathbb{Z}^\nu\rightarrow\mathbb{Z}^\nu$ by the following: $g(t)=t-\tau.$ Then $g(\tau)=\bar{0}$. By Corollary \ref{corollary}.5), 
\[\langle Y^{(k)}_{\tau},Y^{(k)}_{\tau}\rangle_\lambda=\langle Y^{(k)}_{g(\tau)},Y^{(k)}_{g(\tau)}\rangle_\lambda=\langle Y^{(k)}_{\bar{0}},Y^{(k)}_{
\bar{0}}\rangle_\lambda.\] Therefore it is sufficient to prove the theorem only for $\tau=\bar{0}$.
For $n\geqslant 1,s\in\mathbb{Z}^\nu$ denote: 
\[V(n,s)=\sum_{\substack{\gamma\in R^*,
\\
|\gamma|=n,
\\
\gamma\text{ connects }(s)
}} 
\sum_{\substack{t\in \tilde{\gamma},
\\
\gamma\text{ connects }(s,t),
\\
t\neq s
}}
\dfrac{1}{\gamma!}
\langle Q_s,Q_t,\Phi^\backprime_\gamma\rangle_0.\]

First we prove:
\begin{equation}
\text{for any } s\in \mathbb{Z}^\nu, V(n,s)=V(n,\bar{0})=V_n.
\label{eq:V_independent}
\end{equation}

To prove \eqref{eq:V_independent}, fix $s\in \mathbb{Z}^\nu$. Define $f(t)=t-s$ and 
\medskip
\\
$F(\{r_1,r_2\})=(\{f(r_1),f(r_2)\})$. Then $f(s)=\bar{0}$. 
\medskip

Define $h:R^*\rightarrow R^*$ by the following: 
\\
for $\gamma=\{(A_1,q_1),\ldots,(A_m,q_m)\}$, $h(\gamma)=\{(F(A_1),q_1),\ldots,(F(A_m),q_m)\}$. 

Clearly:
\begin{equation}
h(\gamma)!=\gamma!\;;
\label{eq:factorial}
\end{equation}
\begin{multline}
\{\beta\in R^*: |\beta|=n,\beta\text{ connects }(\bar{0})\}
\\
=\{h(\gamma): \gamma\in R^*,|\gamma|=n,\gamma\text{ connects }(s)\};
\label{eq:set_1}
\end{multline}
\begin{multline}
\{r\in \widetilde{h(\gamma)}: r\neq \bar{0}, h(\gamma) \text{ connects }(\bar{0},r)\}
\\
=\{f(t): t\in \gamma,t\neq s, \gamma\text{ connects }(s,t)\}.
\label{eq:set_2}
\end{multline}

Using \eqref{eq:factorial}-\eqref{eq:set_2} and Lemma \ref{lemma:translation_2}, we get:
\begin{multline*}
V_n=V(n,\bar{0})=\sum_{\substack{\beta\in R^*,
\\
|\beta|=n,
\\
\beta\text{ connects }(\bar{0})
}} 
\sum_{\substack{r\in \tilde{\beta},
\\
\beta\text{ connects }(\bar{0},r),
\\
r\neq \bar{0}
}}
\dfrac{1}{\beta!}
\langle Q_{\bar{0}},Q_r,\Phi^\backprime_\beta\rangle_0
\\
=\sum_{\substack{\gamma\in R^*,
\\
|\gamma|=n,
\\
\gamma\text{ connects }(s)
}} 
\sum_{\substack{r\in \widetilde{h(\gamma)},
\\
h(\gamma)\text{ connects }(\bar{0},r),
\\
r\neq \bar{0}
}}
\dfrac{1}{h(\gamma)!}
\langle Q_{\bar{0}},Q_r,\Phi^\backprime_{h(\gamma)}\rangle_0
\\
=\sum_{\substack{\gamma\in R^*,
\\
|\gamma|=n,
\\
\gamma\text{ connects }(s)
}} 
\sum_{\substack{t\in \tilde{\gamma},
\\
\gamma\text{ connects }(s,t),
\\
t\neq s
}}
\dfrac{1}{\gamma!}
\langle Q_{f(s)},Q_{f(t)},\Phi^\backprime_{h(\gamma)}\rangle_0
\\
=\sum_{\substack{\gamma\in R^*,
\\
|\gamma|=n,
\\
\gamma\text{ connects }(s)
}} 
\sum_{\substack{t\in \tilde{\gamma},
\\
\gamma\text{ connects }(s,t),
\\
t\neq s
}}
\dfrac{1}{\gamma!}
\langle Q_s,Q_t,\Phi^\backprime_\gamma\rangle_0=V(n,s).
\end{multline*}

This completes the proof of \eqref{eq:V_independent}.

If $\gamma$ connects $(s)$ and $|\gamma|=n$, then $\tilde{\gamma}$ has at most $n+1$ points. So the second sum in $V(n,s)$ has at most $n$ addends. Using also Lemma \ref{lemma:Eulirean} for $b=(s)$ and Lemma \ref{lemma:estimates} for $b=(s,t)$, we get:
\begin{multline*}
|V(n,s)|\leqslant\sum_{\substack{\gamma\in R^*,
\\
|\gamma|=n,
\\
\gamma\text{ connects }(s)
}}  
\sum_{\substack{t\in \tilde{\gamma},
\\
\gamma\text{ connects }(s,t),
\\
t\neq s
}}
\dfrac{1}{\gamma !}\Big|
\langle Q_s,Q_t,\Phi^\backprime_\gamma\rangle_0\Big|
\\
\leqslant (C_2)^n\cdot n\cdot\dfrac{1}{\gamma !}(C_3)^{2+n}2!\gamma !\;,
\end{multline*}
\begin{equation}
|V(n,s)|\leqslant 2(C_3)^2 n\left(C_2 C_3\right)^n. 
\label{eq:V_estimate}
\end{equation}

Since $V(n,\bar{0})=V_n$, we have by \eqref{eq:V_estimate}:
\[\Big|\sum_{n=1}^{\infty} \lambda^nV_n\Big|\leqslant
2(C_3)^2\sum_{n=1}^{\infty}n|\lambda C_2C_3|^n
=2(C_3)^2\dfrac{|\lambda C_2C_3|}{\left(1-|\lambda C_2C_3|\right)^2}\]
by Lemma \ref{lemma:number_series}.1) for $l=0$.
Since $|\lambda|<\dfrac{1}{2C_2C_3}$, we have $|\lambda C_2C_3|<\dfrac{1}{2}$ and 
\[\Big|\sum_{n=1}^{\infty} \lambda^nV_n\Big|<
\dfrac{|\lambda|2C_2(C_3)^3}{\left(1-\dfrac{1}{2}\right)^2}
=|\lambda|8C_2(C_3)^3\leqslant 1, \text{ so}\] 
\begin{equation}
\Big|\sum_{n=1}^{\infty} \lambda^nV_n\Big|<1.
\label{eq:series_estimate}
\end{equation}

For $n\geqslant 1,s\in\mathbb{Z}^\nu$ denote 
\[W(n,s,k)=\sum_{\substack{\gamma\in R^*,
\\
|\gamma|=n,
\\
\gamma\text{ connects }(s)}} 
\sum_{\substack{t\in G_k^{-1}(\bar{0}),
\\
\gamma\text{ connects }(s,t),
\\
t\neq s}}
\dfrac{1}{\gamma!}
\langle Q_s,Q_t,\Phi^\backprime_\gamma\rangle_0.\]

If $\gamma$ connects $(s,t)$, then $t\in\tilde{\gamma}$, so similarly to \eqref{eq:V_estimate} we get:
\begin{equation}
|W(n,s,k)|\leqslant 2(C_3)^2 n\left(C_2 C_3\right)^n. 
\label{eq:W_estimate}
\end{equation}

Using the definition of renormalization-group, we get:
\begin{multline*}
\langle Y^{(k)}_{\bar{0}},Y^{(k)}_{\bar{0}}\rangle_\lambda=(k^{-\nu/2})^2  \sum_{s\in G_k^{-1}(\bar{0})}\sum_{t\in G_k^{-1}(\bar{0})}\langle Q_{s},Q_{t}\rangle_\lambda
\\
=k^{-\nu}\sum_{s\in G_k^{-1}(\bar{0})}\langle Q_{s},Q_{s}\rangle_\lambda
+k^{-\nu}
\sum_{\substack{s,t\in G_k^{-1}(\bar{0}),
\\
t\neq s
}}
\langle Q_{s},Q_{t}\rangle_\lambda.
\end{multline*}

For any $s\in\mathbb{Z}^\nu$, $\langle Q_{s},Q_{s}\rangle_\lambda=1$ by Corollary \ref{corollary}.3). There are $k^\nu$ points in $G_k^{-1}(\bar{0})$. So 
\begin{multline*}
\langle Y^{(k)}_{\bar{0}},Y^{(k)}_{\bar{0}}\rangle_\lambda=k^{-\nu}\sum_{s\in G_k^{-1}(0)}1+k^{-\nu}
\sum_{\substack{s,t\in G_k^{-1}(\bar{0}),
\\
t\neq s
}}
\langle Q_{s},Q_{t}\rangle_\lambda
\\
=1+k^{-\nu}
\sum_{\substack{s,t\in G_k^{-1}(\bar{0}),
\\
t\neq s
}}
\langle Q_{s},Q_{t}\rangle_\lambda.
\end{multline*}

By Corollary \ref{corollary}.1) for $b=(s,t)$:
\[\langle Y^{(k)}_{\bar{0}},Y^{(k)}_{\bar{0}}\rangle_\lambda
=1+k^{-\nu}
\sum_{\substack{s,t\in G_k^{-1}(\bar{0}),
\\
t\neq s}}
\sum_{n=1}^{\infty} \lambda^n
\sum_{\substack{\gamma\in R^*,
\\
|\gamma|=n,
\\
\gamma\text{ connects }(s,t)}} \dfrac{1}{\gamma !}
\langle Q_s,Q_t,\Phi^\backprime_\gamma\rangle_0.\]
For $\gamma$ connecting $(s,t)$ the length $|\gamma|=n\geqslant 1$, so the summation over $n$ is taken from 1. Let us change the order of summation in this series:
\begin{multline*}
k^{-\nu}\sum_{n=1}^{\infty} \lambda^n\sum_{s\in G_k^{-1}(\bar{0})}
\sum_{\substack{\gamma\in R^*,
\\
|\gamma|=n,
\\
\gamma\text{ connects }(s)}}
\sum_{\substack{t\in G_k^{-1}(\bar{0}),
\\
\gamma\text{ connects }(s,t)
\\
t\neq s}}
\dfrac{1}{\gamma !}
\langle Q_s,Q_t,\Phi^\backprime_\gamma\rangle_0
\\
=\sum_{n=1}^{\infty}\sum_{s\in G_k^{-1}(\bar{0})}
\lambda^nk^{-\nu}W(n,s,k).
\end{multline*}

By \eqref{eq:W_estimate} we have:
\begin{multline*}
\Big|\sum_{s\in G_k^{-1}(\bar{0})}
\lambda^nk^{-\nu}W(n,s,k)\Big|\leqslant
\sum_{s\in G_k^{-1}(\bar{0})}
k^{-\nu}|\lambda|^n|W(n,s,k)|
\\
\leqslant
k^{\nu}k^{-\nu}2(C_3)^2n|\lambda C_2C_3|^n
\leqslant 2(C_3)^2\frac{n}{2^n}.
\end{multline*}

So the new series converges absolutely and uniformly for any $k$. Therefore the series for $\langle Y^{(k)}_{\bar{0}},Y^{(k)}_{\bar{0}}\rangle_\lambda$ converges absolutely and uniformly for any $k$. Hence 
\[\langle Y^{(k)}_{\bar{0}},Y^{(k)}_{\bar{0}}\rangle_\lambda=1+
\sum_{n=1}^{\infty}\lambda^n\sum_{s\in G_k^{-1}(\bar{0})}
k^{-\nu}W(n,s,k)\text{ and }\]
\[\lim_{k\rightarrow\infty}\langle Y^{(k)}_{\bar{0}},Y^{(k)}_{\bar{0}}\rangle_\lambda=1+\sum_{n=1}^{\infty}\lambda^n \lim_{k\rightarrow\infty}A_{n,k}, \text{ where }A_{n,k}=\sum_{s\in G_k^{-1}(\bar{0})}k^{-\nu}W(n,s,k).\]

It remains to show that for any $n\geqslant 1$:
\begin{equation}
\lim_{k\rightarrow\infty}A_{n,k}=V_n.
\label{eq:inner_cube}
\end{equation}
\[\text{Then }\lim_{k\rightarrow\infty}\langle Y^{(k)}_{\bar{0}},Y^{(k)}_{\bar{0}}\rangle_\lambda=1+\sum_{n=1}^{\infty}\lambda^n V_n>0\text{ by \eqref{eq:series_estimate}}.\]
\begin{center}
Proof of \eqref{eq:inner_cube}
\end{center}

Fix $n\geqslant 1$. For $k>2n$ consider a cube in $\mathbb{Z}^\nu$:
\[S_{n,k}=\{r=(r_1,\ldots,r_\nu)\in\mathbb{Z}^\nu\mid n\leqslant r_i\leqslant k-1-n \text{ for each }i=1,2,\ldots,\nu\}.\]

If $\gamma$ connects $(s,t)$, then $\rho(s,t)\leqslant|\gamma|$. So if $\gamma$ connects $(s,t)$, $|\gamma|=n$ and 
\medskip
\\
$s\in S_{n,k}$, then $t\in G_k^{-1}(\bar{0})$ and $t\in\tilde{\gamma}$. Hence for $s\in S_{n,k}$, $W(n,s,k)=V(n,s)=V_n$ by \eqref{eq:V_independent}. 

Denote 
\[R_{n,k}=\sum_{s\in G_k^{-1}(\bar{0})\setminus S_{n,k}}k^{-\nu}W(n,s,k).\]
\[\text{Then }A_{n,k}=\sum_{s\in S_{n,k}}k^{-\nu}W(n,s,k)+R_{n,k}=k^{-\nu}(k-2n)^\nu V_n+R_{n,k},\]
since $S_{n,k}$ contains $(k-2n)^\nu$ points.

Since $G_k^{-1}(\bar{0})\setminus S_{n,k}$ contains $k^\nu-(k-2n)^\nu$ points, by \eqref{eq:W_estimate} we have:
\[|R_{n,k}|\leqslant \sum_{s\in G_k^{-1}(\bar{0})\setminus S_{n,k}}k^{-\nu}|W(n,s,k)|
\leqslant \dfrac{k^\nu-(k-2n)^\nu}{k^\nu} 2(C_3)^2 n(C_2C_3)^n\rightarrow 0\]
as $k\rightarrow\infty$, since $\lim_{k\rightarrow\infty}
\dfrac{k^\nu-(k-2n)^\nu}{k^\nu}
=0$. So
\[\lim_{k\rightarrow\infty}A_{n,k}=V_n
\lim_{k\rightarrow\infty}\dfrac{(k-2n)^\nu}{k^\nu}+\lim_{k\rightarrow\infty}(R_{n,k})=V_n.\]
\end{proof}

\textit{Note:} each $V_n$ is a finite sum; explicit expressions for $V_n$ can be found, which allows to  approximate the limiting variance. In particular, it is easy to show that $V_1=2\nu$ and $V_2=2\nu(2\nu-1)$.

\subsection{Proof of the central limit theorem for Ising model}
\label{proof_main}

\begin{proof}[Proof of Main Theorem (Theorem \ref{theorem:main})]
Constant $C$ was defined in \\Theorem \ref{theorem:variance}. $C\leqslant C_\nu=\dfrac{1}{2C_2C_3}$, the constant from Lemma \ref{lemma:series_N}. 
Suppose $|\lambda|<C$. 

By the definition of renormalization-group, for any $\tau_1,\tau_2,\ldots,\tau_m\in\mathbb{Z}^\nu$:
\begin{multline} 
\langle Y^{(k)}_{\tau_1},Y^{(k)}_{\tau_2},\ldots,Y^{(k)}_{\tau_m}\rangle_\lambda 
\\
=(k^{-\alpha/2})^m  \sum_{t_1 \in G_k^{-1}(\tau_1)}\ldots\sum_{t_m \in G_k^{-1}(\tau_m)}
 \langle Q_{t_1},Q_{t_2},\ldots,Q_{t_m} \rangle_\lambda.
\label{RGm}
\end{multline}

For $m=1$ by Corollary \ref{corollary}.3):
$\langle Y^{(k)}_{\tau}\rangle_\lambda=k^{-\alpha/2}\sum_{t\in G_k^{-1}(\tau)}\langle Q_t\rangle_\lambda=0$. So we have:
\begin{equation}
\text{For any }\tau\in\mathbb{Z}^\nu:\langle Y^{(k)}_{\tau}\rangle_\lambda =0.
\label{eq:expectation}
\end{equation}

1. Suppose $\alpha>\nu$. 

Fix $\tau\in\mathbb{Z}^\nu$. 
In order to show that $Y^{(k)}_\tau \rightarrow 0$ in mean square, it is sufficient to prove that the expectation of  $Y^{(k)}_\tau$  equals 0 (proven in \eqref{eq:expectation}) and its variance tends to 0, that is
\begin{equation}
\lim_{k\rightarrow\infty}\langle Y^{(k)}_{\tau},Y^{(k)}_{\tau}\rangle_\lambda=0.
\label{eq:variance}
\end{equation}

Denote $\varepsilon=\alpha-\nu$.
Then $\varepsilon>0$ and $\alpha=\varepsilon+\nu$.
By \eqref{RGm},
\begin{multline*}
\langle Y^{(k)}_{\tau},Y^{(k)}_{\tau}\rangle_\lambda=k^{-\alpha}\sum_{s\in G_k^{-1}(\tau)}\sum_{t\in G_k^{-1}(\tau)}\langle Q_s,Q_t\rangle_\lambda
\\
=k^{-\varepsilon}k^{-\nu}\sum_{s\in G_k^{-1}(\tau)}\sum_{t\in G_k^{-1}(\tau)}\langle Q_s,Q_t\rangle_\lambda.
\end{multline*}

Theorem \ref{theorem:variance} implies that:
\[\lim_{k\rightarrow\infty}k^{-\nu}\sum_{s\in G_k^{-1}(\tau)}\sum_{t\in G_k^{-1}(\tau)}\langle Q_s,Q_t\rangle_\lambda=Const.\] 
Since $\lim_{k\rightarrow\infty}k^{-\varepsilon}=0$, we get \eqref{eq:variance}.

2. Suppose $\alpha=\nu$. 

First we prove the following formulas \eqref{eq:lim_over2} and \eqref{eq:lim_equal2}.
\begin{equation}
\text{For }m\geqslant 3 \text{ and any }\tau_1,\tau_2,\ldots,\tau_m\in\mathbb{Z}^\nu: 
\lim_{k\rightarrow\infty}\langle Y^{(k)}_{\tau_1},Y^{(k)}_{\tau_2},\ldots,Y^{(k)}_{\tau_m}\rangle_\lambda =0.
\label{eq:lim_over2}
\end{equation}
\begin{equation}
\text{For any }\tau,\theta\in\mathbb{Z}^\nu,\text{ such that }\tau\neq\theta:\lim_{k\rightarrow\infty}\langle Y^{(k)}_{\tau},Y^{(k)}_{\theta}\rangle_\lambda =0.
\label{eq:lim_equal2}
\end{equation}

Assume \eqref{eq:lim_over2} and \eqref{eq:lim_equal2} are proven. The formulas \eqref{eq:expectation},  \eqref{eq:lim_over2}, \eqref{eq:lim_equal2} and Theorem \ref{theorem:variance} determine the limits of all the semi-invariants $\langle Y^{(k)}_{\tau_1},Y^{(k)}_{\tau_2},\ldots,Y^{(k)}_{\tau_m}\rangle_\lambda$ as $k\rightarrow\infty$. All of the limiting semi-invariants equal 0, except the variances. Therefore the random variables $Y^{(k)}_{\tau_1},Y^{(k)}_{\tau_2},\ldots,Y^{(k)}_{\tau_m}$ converge in distribution as $k\rightarrow\infty$ to an independent multivariate normal random vector, due to Lemma \ref{lemma:Carleman}. The statement about the expectation and variance of the limiting distribution follows from \eqref{eq:expectation} and Theorem \ref{theorem:variance}. So it remains to prove \eqref{eq:lim_over2} and \eqref{eq:lim_equal2}.

\begin{center}
Proof of \eqref{eq:lim_over2}
\end{center}

Fix $m\geqslant 3$ and $\tau_1,\ldots,\tau_m\in\mathbb{Z}^\nu$. By \eqref{RGm} and Corollary \ref{corollary}.1) we have: 
\begin{multline*}
\langle Y^{(k)}_{\tau_1},\ldots,Y^{(k)}_{\tau_m}\rangle_\lambda
\\
=k^{-m\nu/2}
\sum_{t_1 \in G_k^{-1}(\tau_1)}
\ldots\sum_{t_m \in G_k^{-1}(\tau_m)}\sum_{n=0}^{\infty}\lambda^n
\sum_{\substack{\gamma\in R^*,
\\
|\gamma|=n,
\\
\gamma\text{ connects }b}} 
\dfrac{1}{\gamma !}\langle Q^,_b,\Phi^\backprime_\gamma\rangle_0.
\end{multline*}
Here each $b=(t_1,t_2,\ldots,t_m)$.
Let us change the order of summation and denote:
\begin{equation}
A_k=k^{-m\nu/2}\sum_{n=0}^{\infty}B_{n,k},
\text{ where} 
\label{eq:A_k}
\end{equation}
\[B_{n,k}=\lambda^n
\sum_{t_1 \in G_k^{-1}(\tau_1)}
\sum_{\substack{\gamma\in R^*,
\\
|\gamma|=n,
\\
\gamma\text{ connects }(t_1)}}  \dfrac{1}{\gamma !}
\sum_{t_2 \in G_k^{-1}(\tau_2)}
\ldots\sum_{t_m \in G_k^{-1}(\tau_m)}\Big|\langle Q^,_b,\Phi^\backprime_\gamma\rangle_0\Big|.\]
\medskip

It is sufficient to show that the series \eqref{eq:A_k} converges and $\lim_{k\rightarrow\infty}A_k=0$. Indeed, that implies that the original series for $\langle Y^{(k)}_{\tau_1},\ldots,Y^{(k)}_{\tau_m}\rangle_\lambda$ absolutely converges, the order of summation does not matter and \[\lim_{k\rightarrow\infty}\langle Y^{(k)}_{\tau_1},\ldots,Y^{(k)}_{\tau_m}\rangle_\lambda=0.\]

Since $t_1 \in G_k^{-1}(\tau_1)$, there are $k^\nu$ choices for $t_1$. By Lemma \ref{lemma:Eulirean}, there are at most $(C_2)^n$ families $\gamma$ of length $n$ that connect $(t_1)$.

Next we fix $t_1$ and a family $\gamma$ of length $n$ connecting $(t_1)$. Then there are at most $n+1$ elements in $\tilde{\gamma}$ and $t_i\in \tilde{\gamma}$ $(i=2,\ldots,m)$. So there are at most $(n+1)^{m-1}$ choices for $(t_2,\ldots,t_m)$.

By Lemma \ref{lemma:estimates}, for each $b$ and $\gamma$ with $|\gamma|=n$: 
$|\langle Q^,_b,\Phi^\backprime_\gamma\rangle_0|\leqslant 
(C_3)^{m+n}\;m!\;\gamma!$ 
\medskip
\\
So each $B_{n,k}\leqslant |\lambda|^nk^\nu(C_2)^n(n+1)^{m-1}
(C_3)^{m+n}\;m!$ and 
\begin{equation}
B_{n,k}\leqslant k^\nu
(C_3)^{m}\;m!(n+1)^{m-1}|\lambda C_2C_3|^n.
\label{eq:inequality_21}
\end{equation}

Since $|\lambda|<\dfrac{1}{2C_2C_3}$, then $|\lambda C_2C_3|<\dfrac{1}{2}$. So by \eqref{eq:inequality_21} and Lemma \ref{lemma:number_series}.2), the series $\sum_{n=0}^\infty B_{n,k}$ converges, the series \eqref{eq:A_k} converges and 
\begin{multline*}
A_k\leqslant k^{-m\nu/2}k^\nu
(C_3)^{m}\;m!\sum_{n=0}^\infty (n+1)^{m-1}|\lambda C_2C_3|^n
\\
\leqslant k^{\frac{(2-m)\nu}{2}}(C_3)^{m}\;m!\dfrac{m!}{(1-|\lambda C_2C_3|)^{m+1}}
=k^{\frac{(2-m)\nu}{2}}\dfrac{(C_3)^{m}(m!)^2}{(1-|\lambda C_2C_3|)^{m+1}}.
\end{multline*}

Since $C_2, C_3$ and $m\geqslant 3$ are fixed we have:
\[0\leqslant \lim_{k\rightarrow\infty}A_k\leqslant \dfrac{(C_3)^{m}(m!)^2}{(1-|\lambda C_2C_3|)^{m+1}}\lim_{k\rightarrow\infty}k^{\frac{(2-m)\nu}{2}}=0\text{ and }\lim_{k\rightarrow\infty}A_k=0.\]

\begin{center}
Proof of \eqref{eq:lim_equal2}
\end{center}

Fix $\tau,\theta\in\mathbb{Z}^\nu$, $\tau\neq\theta$. We consider four cases. 

\textit{Case 1:} the first coordinate of $\tau$ equals 0 and the first coordinate of $\theta$ is negative. 

Clearly, for any $t=(t_1,\ldots,t_\nu)\in G_k^{-1}(\tau)$ we have: $0\leqslant t_1\leqslant k-1$. Similarly, for any $s\in G_k^{-1}(\theta)$, $s_1\leqslant -1$.
We introduce cross-sections of the cube $G_k^{-1}(\tau)$:
\[D_l=\left\lbrace t=(t_1,\ldots,t_\nu)\in G_k^{-1}(\tau)\mid t_1=l\right\rbrace , l=0,1,\ldots,k-1.\]

Clearly, $G_k^{-1}(\tau)=\bigcup_{l=0}^{k-1}D_l$. Next we show:
\begin{equation}
\text{if }t\in D_l,s\in G_k^{-1}(\theta)\text{ and }\gamma \text{ connects }(t,s), \text{ then }|\gamma|\geqslant l+1.
\label{eq:cross-section}
\end{equation}

For $t\in D_l$ we have $t_1=l$. For $s\in G_k^{-1}(\theta)$ we have $s_1\leqslant -1$. So the distance between such $t$ and $s$ is not less than $l+1$. If a family $\gamma$ connects $(t,s)$, then the length of $\gamma$ is at least $l+1$. This proves \eqref{eq:cross-section}.

By \eqref{RGm}:
\[\langle Y^{(k)}_{\tau},Y^{(k)}_{\theta}\rangle_\lambda=k^{-\nu}\sum_{t\in G_k^{-1}(\tau)}\sum_{s \in G_k^{-1}(\theta)}
 \langle Q_{t},Q_{s}\rangle_\lambda=
 k^{-\nu}\sum_{l=0}^{k-1}
 \sum_{t\in D_l}\sum_{s \in G_k^{-1}(\theta)}
 \langle Q_{t},Q_{s}\rangle_\lambda
 \]
and by Corollary \ref{corollary}.1):
\[\langle Y^{(k)}_{\tau},Y^{(k)}_{\theta}\rangle_\lambda=k^{-\nu}\sum_{l=0}^{k-1}
 \sum_{t\in D_l}\sum_{s \in G_k^{-1}(\theta)}\sum_{n=0}^\infty \lambda^n\sum_{\substack{\gamma\in R^*,
\\
|\gamma|=n,
\\
\gamma\text{ connects }(t,s)}} \dfrac{1}{\gamma !}
\langle Q_{t},Q_{s},\Phi^\backprime_\gamma\rangle_0.\]
By \eqref{eq:cross-section}, the sum over $n$ can be taken from $l+1$ instead of 0:
\[\langle Y^{(k)}_{\tau},Y^{(k)}_{\theta}\rangle_\lambda=k^{-\nu}\sum_{l=0}^{k-1}
 \sum_{t\in D_l}\sum_{s \in G_k^{-1}(\theta)}\sum_{n=l+1}^\infty \sum_{\substack{\gamma\in R^*,
\\
|\gamma|=n,
\\
\gamma\text{ connects }(t,s)}}\dfrac{\lambda^n}{\gamma !}
\langle Q_{t},Q_{s},\Phi^\backprime_\gamma\rangle_0.\]

Let us change the order of summation and denote:
\begin{equation}
S_k=k^{-\nu}\sum_{l=0}^{k-1}
 \sum_{t\in D_l}\sum_{n=l+1}^\infty d_{n,k},\text{ where }
\label{eq:S_k}
\end{equation}
\[d_{n,k}=\sum_{\substack{\gamma\in R^*,
\\
|\gamma|=n,
\\
\gamma\text{ connects }(t)}}
\sum_{\substack{s \in G_k^{-1}(\theta),
\\
\gamma\text{ connects }(t,s)}}
\Big|\dfrac{\lambda^n}{\gamma!}\langle Q_{t},Q_{s},\Phi^\backprime_\gamma\rangle_0\Big|.\]

It is sufficient to show that the series \eqref{eq:S_k} converges and $\lim_{k\rightarrow\infty}S_k=0$. Indeed, that implies that the original series for $\langle Y^{(k)}_{\tau},Y^{(k)}_{\theta}\rangle_\lambda$ absolutely converges, the order of summation does not matter and $\lim_{k\rightarrow\infty}\langle Y^{(k)}_{\tau},Y^{(k)}_{\theta}\rangle_\lambda =0.$ 
\medskip

By Lemma \ref{lemma:estimates}, if $|\gamma|=n$, then $\Big|\langle Q_{t},Q_{s},\Phi^\backprime_\gamma\rangle_0\Big|\leqslant (C_3)^{n+2}2!\gamma!=2(C_3)^{n+2}\gamma!$
\medskip

Assume $\gamma\in R^*$ and $t\in D_l$ are fixed, $|\gamma|=n$ and $\gamma$ connects $(t)$. Then $|\tilde{\gamma}|\leqslant n+1$. In order for $\gamma$ to connect $(t,s)$, the point $s$ should be among the elements of $\tilde{\gamma}$ with negative first coordinates, and $\tilde{\gamma}$ should contain points with first coordinates $l,l-1,l-2,\ldots,0$, so at least $l+1$ points of $\tilde{\gamma}$ have non-negative first coordinates. Therefore there are at most $n+1-(l+1)=n-l$ choices for $s$. By Lemma \ref{lemma:Eulirean}, there are at most $(C_2)^n$ families $\gamma$ with $|\gamma|=n$ that connect $(t)$. So 
\[d_{n,k}\leqslant (C_2)^n(n-l)\dfrac{|\lambda|^n}{\gamma!}2(C_3)^{n+2}\gamma!=2C_3^2(n-l)|\lambda C_2C_3|^n.\]

Since $|\lambda|<\dfrac{1}{2C_2C_3}$, then $|\lambda C_2C_3|<\dfrac{1}{2}$. By Lemma \ref{lemma:number_series}.1) the series \eqref{eq:S_k} converges and since $D_l$ contains $k^{\nu-1}$ points, we have:
\begin{multline*}
0\leqslant S_k
\leqslant 
2(C_3)^2k^{-\nu}\sum_{l=0}^{k-1}k^{\nu-1}\sum_{n=l+1}^\infty(n-l)|\lambda C_2C_3|^n
\\
=\dfrac{2(C_3)^2}{k}\sum_{l=0}^{k-1}\dfrac{|\lambda C_2C_3|^{l+1}}{\left(1-|\lambda C_2C_3| \right)^2}
\leqslant \dfrac{2(C_3)^2}{k\left(1-|\lambda C_2C_3| \right)^2}\sum_{l=0}^{\infty}|\lambda C_2C_3|^{l+1}
\\
=\dfrac{2(C_3)^2}{k\left(1-|\lambda C_2C_3| \right)^2}\cdot\dfrac{|\lambda C_2C_3|}{1-|\lambda C_2C_3|}.
\end{multline*}
Therefore $\lim_{k\rightarrow\infty}S_k=0$.

\textit{Case 2:} the first coordinate of $\tau$ is greater than the first coordinate of $\theta$.

Denote $a=(\tau_{1},0,0,\ldots,0)$. Then $a\in \mathbb{Z}^\nu$. Define $g$ by: $g(t)= t-a$
 and denote $\tau'=g(\tau)$, $\theta'=g(\theta)$. 
 
Then by Corollary \ref{corollary}.5), $\langle Y^{(k)}_{\tau},Y^{(k)}_{\theta}\rangle_\lambda=\langle Y^{(k)}_{\tau'},Y^{(k)}_{\theta'}\rangle_\lambda$; $\tau'$ and $\theta'$ satisfy the conditions of Case 1. Thus, Case 2 is reduced to Case 1.

\textit{Case 3}: the first coordinate of $\tau$ is less than the first coordinate of $\theta$.

This is reduced to Case 2 by interchanging $\tau$ and $\theta$.

\textit{Case 4:} the general case.

Since $\tau\neq\theta$, they should differ in at least one coordinate, for example, in $j$-th coordinate. The proof is obtained by applying the proofs in Cases 1-3 to $j$-th coordinates instead of the first coordinates. This completes the proof of \eqref{eq:lim_equal2} and the proof of the theorem. 
\end{proof}

\section{Discussion}

In this paper we prove a generalization of the central limit theorem to a random field transformed by renormalization group, in Ising model with no external field and with a constant strength of interaction. We show that as $k\rightarrow\infty$ the resulting random fields $Y_s^{(k)}$ converge in distribution to an independent random field with Gaussian distribution. We find the limits of all semi-invariants of $Y_s^{(k)}$ as $k\rightarrow\infty$ and apply Carleman's theorem. In particular, we show that all the semi-invariants, except the variances, tend to 0. In Theorem \ref{theorem:variance} we find an explicit expression for the limiting variance. In order to find the limiting semi-invariants, we derive estimations of the semi-invariants of the original random field with respect to Gibbs measure. 

We modify the techniques of estimating semi-invariants in Ising model from \cite{Ma80} and \cite{MM91} and apply it to derive a useful expression for semi-invariants with respect to the limiting Gibbs measure in Corollary \ref{corollary}.1). We provide a more transparent proof under more general conditions for the inequality about the number of links in a set with a symmetric binary relation (Theorem \ref{theorem:estimation}). In this theorem and the lemmas about estimations of  semi-invariants, as well as in the main theorem, we derive explicit expressions for the estimation constants. 

A possible direction for future research is generalization of our theorem to other types of Ising model and other types of distribution of the original random field. 

\bibliographystyle{asl}
\bibliography{Ilias}

\end{document}